\documentclass{amsart}

\usepackage{amsmath}
\usepackage{amssymb}
\usepackage[margin=1in]{geometry}
\usepackage{bm}
\usepackage{amsthm}
\usepackage{enumerate}
\usepackage{graphicx}
\usepackage{psfrag}
\usepackage{multicol}
\usepackage[dvipsnames]{xcolor}
\usepackage{url}
\usepackage{subcaption}
\usepackage{algpseudocode}
\usepackage{mathtools}
\usepackage{xy}
\input xy
\xyoption{all}
\usepackage{tikz-cd}
\usepackage{wasysym}
\usepackage{hyperref}
\usepackage{ytableau}

\newlength{\cellsize}
\ytableausetup{boxsize=3ex}

\newcommand\tableau[1]{
\vcenter{
\let\\=\cr
\baselineskip=-16000pt
\lineskiplimit=16000pt
\lineskip=0pt
\halign{&\tableaucell{##}\cr#1\crcr}}}

\newcommand{\tableaucell}[1]{{%
\def \arg{#1}\def \void{}%
\ifx \void \arg
\vbox to \cellsize{\vfil \hrule width \cellsize height 0pt}%
\else
\unitlength=\cellsize
\begin{picture}(1,1)
\put(0,0){\makebox(1,1)[c]{$#1$}}
\put(0,0){\line(1,0){1}}
\put(0,1){\line(1,0){1}}
\put(0,0){\line(0,1){1}}
\put(1,0){\line(0,1){1}}
\end{picture}%
\fi}}

\newtheorem{thm}{Theorem}[section]
\newtheorem{lemma}[thm]{Lemma}

\newtheorem{cor}[thm]{Corollary}
\newtheorem{prop}[thm]{Proposition}
\newtheorem{conj}[thm]{Conjecture}

\theoremstyle{definition}

\newtheorem{example}[thm]{Example}
\newtheorem{remark}[thm]{Remark}
\newtheorem{defn}[thm]{Definition}

\numberwithin{equation}{section}

\newcommand{\CC}{\mathbb{C}}

\newcommand{\ZZ}{\mathbb{Z}}

\newcommand{\cal}[1]{{\mathcal{#1}}}

\newcommand{\gr}[1]{\textcolor{ForestGreen}{#1}}
\newcommand{\bl}[1]{\textcolor{blue}{#1}}
\newcommand{\red}[1]{\textcolor{red}{#1}}

\DeclareMathOperator{\pr}{pr}
\newcommand{\rw}{\mathsf{rw}}
\newcommand{\cw}{\mathsf{cw}}
\DeclareMathOperator{\charge}{charge}

\DeclareMathOperator{\SSYT}{SSYT}
\DeclareMathOperator{\jdt}{jdt}
\DeclareMathOperator{\Ind}{Ind}

\newcommand{\scharge}{\mathsf{scharge}}
\DeclareMathOperator\trace{tr}
\DeclareMathOperator{\character}{char}
\DeclareMathOperator{\diag}{diag}
\newcommand{\row}{\mathsf{row}}
\newcommand{\col}{\mathsf{col}}
\begin{document}

\mbox{}
\title[KR Dual Equivalence Graphs]{Kirillov-Reshetikhin Dual Equivalence Graphs}

\author[McDonough]{Joseph McDonough}
\address{School of Mathematics\\University of Minnesota\\Minneapolis, MN 55455}
\email{mcdo1248@umn.edu}

\author[Pylyavskyy]{Pavlo Pylyavskyy}
\address{School of Mathematics\\University of Minnesota\\Minneapolis, MN 55455}
\email{ppylyavs@umn.edu}

\author[Wang]{Shiyun Wang}
\address{School of Mathematics\\University of Minnesota\\Minneapolis, MN 55455}
\email{wang8406@umn.edu}

\date{\today}

\begin{abstract}
Let $U$ be a tensor product of highest weight modules of $GL_n(\mathbb C)$ corresponding to multiples of fundamental weights (i.e. rectangles). We consider three ways to stratify $U^{\otimes k}$ into components: using isotypic components of the cyclic action on tensor factors, using a generalization of the charge statistic, and using certain generalizations of Assaf's dual equivalence graphs. We conjecture that all three ways coincide, and we prove that the latter two ways coincide. The Kirillov-Reshetikhin dual equivalence graphs (KR DEGs) we introduce for this purpose are defined on $0$-weight spaces of tensor products of Kirillov-Reshetikhin crystals. They generalize Kazhdan-Lusztig dual equivalence graphs (KL DEGs) that previously appeared in the study of Kazhdan-Lusztig cells in affine type A. While the tensor products of Kirillov-Reshetikhin crystals are connected as affine crystals, the KR DEGs in general are not. 
\end{abstract}

\maketitle

\section{Introduction}
Let $V = \mathbb C^n$ and let $\phi_U: GL(V) \rightarrow GL(U)$ be a polynomial representation of $GL_n(\mathbb C)$. Consider the $GL(V)$ representation $U^{\otimes k} = U \otimes U \otimes \cdots \otimes U$ and the action of cyclic group $C_k$ on it by cycling tensor factors. One can split $U^{\otimes k}$ into isotypic components of $C_k$, and since the two actions commute, each component carries a representation of $GL(V)$. Thus, we stratify $U^{\otimes k}$ into $k$ components as a $GL_n(\mathbb C)$ representation. To a representation $\phi_U$, one can associate a character $\character_U(x) = \trace(\phi_U(\diag(x)))$, which is a symmetric function given by the trace of the image of a diagonal matrix with diagonal entries $x_i$ for each $i \in [n]$. Recall that cyclic characters $\ell_k^{(m)} = ch(\Ind_{C_k}^{S_k}(\chi^m))$ are Frobenius characters of the inductions of irreducible characters $\chi^m$ of $C_k$ to $S_k$. Cyclic characters were studied by Foulkes \cite{Foulkes72}, and further work connecting them to the charge statistic was done by Kraskiewicz and Weyman \cite{kraskiewiczweyman01} as well as Lascoux, Leclerc and Thibon \cite{LLT94}. It can be shown that characters of the above components of $U^{\otimes k}$ can be obtained as plethysms $\ell_k^{(m)}[\character_U(x)]$. 

Assume now $U$ itself is a tensor product of highest weight modules of $GL_n(\mathbb C)$, where each factor is given by a multiple of a fundamental weight. In other words, $U$ is a tensor product of highest weight modules of rectangular shapes. By a slight abuse of terminology, we shall also refer to them as Kirillov-Reshetikhin modules. There are two other ways to stratify $U^{\otimes k}$ into $k$ components - one using a generalization of the charge statistic introduced by Shimozono \cite{shimozono02} modulo $k$, and one using a certain generalization of Assaf's dual equivalence graphs. We conjecture that all three ways to stratify $U^{\otimes k}$ coincide, and we prove that the latter two ways do. In special cases, the fact that the first two ways coincide is known due to Lascoux, Leclerc, and Thibon \cite{LLT94}, and Iijima \cite{Iijima13}. A related question of splitting a tensor square of a representation into symmetric and antisymmetric parts has also been studied by Carr\'e and Leclerc \cite{carre1995} and Maas-Gari\'epy and T\'etreault \cite{mgt22}. 

Kirillov-Reshetikhin crystals (KR crystals) and their tensor products (which we denote $\otimes$KR crystals) are an important family of affine crystals, which arise in the representation theory of quantum groups \cite{KR1990}. The $\otimes$KR crystals are indexed by sequences of rectangular Young diagrams, and can be realized as the corresponding type A crystal with the extra operators $f_n$ and $e_n$ (often denoted $f_0$ and $e_0$). These action of these operators can be computed explicitly using promotion on tableaux \cite{shimozono02}. Extensive work has been done connecting $\otimes$KR crystals to Demazure crystals, see \cite{shimozono02,naoi13,schillingtingley11,lnsss14,lnsss16,lnsss17}. $\otimes$KR crystals have a grading by the global energy function, which is computed using combinatorial $R$-matrices and the local energy function between tensor factors. Nakayashiki and Yamada \cite{NY1997} proved that the energy function on tensor products of single columns, considered as $\otimes$KR crystals, equals the cocharge of the corresponding semistandard tableau (originally defined by Lascoux and Sch{\"u}tzenberger \cite{LS1978, LS1984}) obtained by insertion via the Robinson-Schensted-Knuth (RSK) correspondence. This result was generalized to more general KR crystals by \cite{Shimozono1995, shimozono02, schillingtingley11, lnsss14, lnsss16}.

Now we take $U$ as a tensor product of KR modules over $GL_n(\mathbb C)$. The character $\character_{U^{\otimes k}}(x)$ is a product of rectangular Schur functions. Associating with $U^{\otimes k}$ a $\otimes$KR crystal, one can define the energy/charge function on it, and further stratify $U^{\otimes k}$ into $k$ disjoint components by the residue of charge modulo $k$. 

Dual equivalence graphs (DEGs) were introduced by Assaf \cite{assaf15} as a tool to prove Schur positivity of symmetric functions. They have deep connections to crystal graphs \cite{assaf08crystals, brauneretal25, chmutovlewispylyavskyy23}. A DEG is a finite graph with vertices given by the standard Young tableaux of a fixed shape, and edges given by  dual Knuth moves. One can also construct DEGs via dual Knuth moves on permutations, which preserve the recording tableau of the permutation under the RSK correspondence. Certain generating functions over DEGs are the Schur functions. Assaf's work then provides axiomatization of DEGs, allowing one to prove Schur positivity by verifying a set of local axioms on graphs.  A $W$-graph is a combinatorial graph encoding the representations of Hecke algebra associated with a Coxeter group $W$. $W$-graphs were introduced in the pioneering work of Kazhdan and Lusztig, further notion of molecules as certain connected components of $W$-graphs was introduced by Stembridge in  \cite{Stembridge08}, \cite{Stembridge12}. Chmutov \cite{chmutov15} showed that in type A molecules are exactly the DEGs as described above. Further, Chmutov, Pylyavskyy, and Yudovina \cite{chmutovpylyavskyyyudovina18} introduced KL DEGs  - an affine analogue of the DEGs, using affine permutations and tabloids. 

The main combinatorial object we introduce in this paper is the Kirillov-Reshetikhin dual equivalence graph (KR DEG). they are defined on the $0$-weight space of a tensor product of KR crystals, in a manner analogous to Assaf's construction in the type A case. We show that, similarly to Assaf's work, the edges of KR DEGs can be realized by commutators of consecutive crystal operators. Perhaps surprisingly, KR DEGs are not always connected, even though the corresponding $\otimes$KR crystals are. This was observed in \cite{chmutovlewispylyavskyy23} for KL DEGs, which are a special case of KR DEGs, see Remark \ref{KLDEG}. The connected components of the corresponding KR DEG provide another natural stratification of $U^{\otimes k}$. The main result of this paper, Theorem \ref{t:connectedcomponents}, shows that this stratification agrees with the stratification based on charge modulo $k$. We note here that, as far as we can tell, KR DEGs are unrelated to the affine dual equivalence graphs defined in \cite{assafbilley12}, which used starred strong tableaux to study $k$-Schur functions.

Our paper is organized as follows: In Section \ref{sec:background}, we introduce definitions and notations, as well as important algorithms and facts on representation of $GL(V)$, tableaux, crystal graphs and dual equivalence graphs. In Section \ref{sec:krdeg}, we construct the KR DEGs from the 0 weight space of $\otimes$KR crystals, and then establish their connections to the affine crystals. In Section \ref{sec:krdegccs}, we study the structure of the KR DEGs and prove that the number of connected components is exactly the $\gcd$ of the multiplicities of each rectangle. Furthermore, the elements in each connected component are determined by the charge statistic. In Section \ref{sec:characters} we discuss the characters associated to connected components of KR DEGs, and conjecture a formula for them involving plethysms with cyclic characters.

\subsection{Acknowledgements}
The authors would like to thank Bernard Leclerc, Vic Reiner, Mark Shimozono and Jean-Yves Thibon for helpful discussions and comments. P.P. was partially supported by DMS-1949896 and SFI-MPS-SFM-00011393.

\section{Background} \label{sec:background}

\subsection{Tableau Combinatorics}
For a partition $\lambda=(\lambda_1, \lambda_2, \ldots)$, a \emph{semistandard Young tableau} (SSYT) $T$ of shape $\lambda$ is a filling of the cells of the Young diagram of $\lambda$ by elements of $[n] = \{1, \ldots, n\}$ that increases weakly in every row and increases strictly in every column. If $T$ uses each number in $[n]$ only once, we say $T$ has a standard filling. For a pair of partitions $\lambda=(\lambda_1, \lambda_2, \ldots)$ and $\mu=(\mu_1, \mu_2, \ldots)$, satisfying $\mu_i \leq \lambda_i$ for each $i$, the \emph{skew} semistandard (resp. standard) tableaux are defined in the natural way in terms of the skew shape $\lambda / \mu$. Throughout the paper, we use English notation for partitions and tableaux.

Let $T$ be a (skew) tableau. The \emph{row-reading word} of $T$, denoted by $\rw(T)$, is the word obtained by concatenating the entries of $T$ row by row from bottom to top, reading from left to right in each row. The \emph{column-reading word} of $T$, denoted by $\cw(T)$, is the word obtained by concatenating the entries of $T$ column by column from left to right, reading each column from bottom to top.

\begin{example}
    Let $\lambda/\mu=(4,3,2)/(1)$. For $T=\tableau{&1&2&6\\3&3&4\\5&6}$, $\rw(T)=56334126$, and $\cw(T)=53631426.$
\end{example}

Let $T$ be a skew tableau of shape $\lambda/\mu$, and let $a$ be a box that shares at least one edge with $T$ such that adding $a$ to $T$ results in a valid skew shape. We call $a$ an inner (resp. outer) box if $a$ shares lower or right (resp. upper or left) edges with $T$. The \emph{jeu de taquin slide} of $T$ in terms of $a$, denoted by $\jdt_a(T)$, is defined in the following steps: First, mark box $a$ with $*$. If the inner (resp. outer) box $a$ shares only one edge with the box $a_1$ of $T$, then slide $a_1$ into $*$. If $a$ shares edges with two boxes of $T$, choose the one with a smaller (resp. larger) entry and slide the entry into $*$. If the two adjacent entries are equal, then slide the lower one upward (resp. slide the higher one downward). Third, continue the same procedure until $*$ is moved to the outer (resp. inner) boundary. Lastly, remove $*$ to obtain a valid skew tableau $\jdt_a(T)$. 

\begin{example}
    For $T=\tableau{a&1&2&6\\1&3&4\\5&6&b}$ with the inner box $a$ and the outer box $b$ marked. Then $\jdt_a(T)$ is obtained by the sequence of jeu de taquin slides 
    \[\tableau{*&1&2&6\\1&3&4\\5&6}\rightarrow \tableau{1&1&2&6\\*&3&4\\5&6}\rightarrow \tableau{1&1&2&6\\3&*&4\\5&6}\rightarrow \tableau{1&1&2&6\\3&4\\5&6}.\]
    Similarly, $\jdt_b(T)=\tableau{&1&2&6\\&3&4\\1&5&6}$.
\end{example}


\begin{defn}\label{promotion}
Let $T$ be a semistandard tableau with entries at most $n$. The \emph{promotion operator} $\pr$ on $T$ with respect to $n$ is a method of obtaining a new semistandard tableau $\pr(T)$, consisting of the following steps:
\begin{enumerate}
    \item[(1)] Mark all $n$'s in $T$ as outer boxes $*$.
    \item[(2)] Compute $\jdt_n$ from bottom to top, and add the entry $0$ in the empty spaces on the inner boundary. 
    \item[(3)] Increase all entries by $1$.
\end{enumerate}
\end{defn}

\begin{example}
Let $n=6, T=\tableau{1&1&2&6\\3&4\\5&6}$. Then $\pr(T)$ is obtained as follows: 
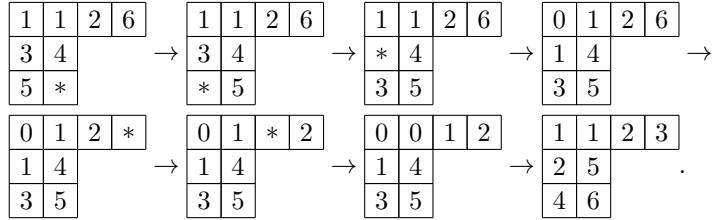
\begin{align*}
&\tableau{1&1&2&6\\3&4\\5&*}\rightarrow \tableau{1&1&2&6\\3&4\\*&5}\rightarrow \tableau{1&1&2&6\\*&4\\3&5}\rightarrow \tableau{0&1&2&6\\1&4\\3&5}\rightarrow \\
&\tableau{0&1&2&*\\1&4\\3&5}\rightarrow \tableau{0&1&*&2\\1&4\\3&5}\rightarrow \tableau{0&0&1&2\\1&4\\3&5}\rightarrow \tableau{1&1&2&3\\2&5\\4&6}.
\end{align*}
\end{example}

Now, we briefly recall the \emph{row insertion} algorithm of the RSK correspondence, which we will use later to define the \emph{charge} of two tableaux.
Let $T$ be a semistandard tableau. Row inserting a positive integer $k$ into $T$, denoted by $T \leftarrow k$, is defined as follows: Find the left-most cell of the first row of $T$ containing a value strictly larger than $k$. If no such cell exists, append a cell filled with $k$ to the end of the row. Otherwise, replace that cell with $k$, and insert the value that used to be there into the next row using the same procedure.

\begin{example}
    Let $T=\tableau{1&1&2&4&5\\3&3&4&6\\5&7}$. Then $T\leftarrow 2=\tableau{1&1&2&\textbf{2}&5\\3&3&4&\textbf{4}\\5&\textbf{6}\\7}$, where the bumped positions are highlighted in bold.
\end{example}

\subsection{Crystal Graphs}

\begin{defn}
  Consider a word $w$ on the alphabet $[n]$. For $i \in [n-1]$, define the \emph{raising} and \emph{lowering} operators $e_i$ and $f_i$ on $w$ as follows:
\begin{enumerate}[\quad 1.]
    \item[(1)] In the subword of $w$ consisting of only $i$'s and $(i+1)$'s, replace $i+1$ with `` ( '' and $i$ with `` ) ''.
    \item[(2)] Remove matched parentheses.
    \item[(3)] If applying $f_i$, replace the right-most unmatched $i$ with $i+1$. If applying $e_i$, replace the left-most unmatched $i+1$ with $i$. The word remains unchanged if no such unmatched entries exist.
\end{enumerate}
\end{defn}
It is clear that $e_i$ and $f_i$ are inverse operations when they are defined.

\begin{example}
    Let $w=1345264453455$ and $i=4,$ we replace all $4$ with `` ) '' and $5$ with `` ( '' to obtain the sequence of parentheses $\textbf{)}()\textbf{)}()\textbf{((}$. Then we match the parentheses in the usual way, and the unmatched ones are highlighted in bold. Thus, $f_i(w)=1345264\textbf{5}53455$ and $e_i(w)=13452644534\textbf{4}5$.
\end{example}

We can extend this definition to a semistandard tableau $T$ by applying the raising and lowering operators to $\rw(T)$ (resp. $\cw(T)$), and changing the cell of $T$ that corresponds to the letter that was changed. For simplicity, we write $f_i(T)$ and $e_i(T)$. It is well known (see, for instance, \cite[Chapter 3]{bumpschilling17}) that the set of row (resp. column) reading words of semistandard tableaux is closed under $e_i$ and $f_i$ for $i \in [n-1]$.

\begin{defn}\label{crystalg}
    Let $\lambda$ be a partition on alphabet $[n]$. The type $A_{n-1}$ \emph{crystal graph} $B_\lambda$ is a finite directed graph with vertices given by $\SSYT(\lambda)$ and edges colored by $\{1,2,\cdots,n-1\}$. In particular, we have an $i$-colored edge pointing from tableau $T$ to $T'$ if $f_i(T) = T'$. See Figure \ref{KR crystal} for an example.
\end{defn}

For a rectangular partition $\lambda=R=(s^r)$, the type $A_{n-1}$ crystal graph $B_{\lambda}$ admits a type $\widetilde{A}_{n-1}$ crystal structure. The underlying vertices are the same, as are the normal crystal operators $e_i$ and $f_i$ for $i \in [n-1]$. The crystal operators $e_n, f_n$ for the affine node were given explicitly by Shimozono \cite{shimozono02} as
\[
  e_n = \pr^{-1} \circ\ e_1 \circ \pr \quad \text{and} \quad  f_n = \pr^{-1} \circ f_1 \circ \pr
\]
The cyclic symmetry of the $\widetilde{A}_{n-1}$ Dynkin diagram also implies that, more generally, $e_{i+1} = \pr^{-1} \circ\ e_i \circ \pr$ and $f_{i+1} = \pr^{-1} \circ\ f_i \circ \pr$. The classical crystal $B_{R}$ with this added structure is the \emph{Kirillov-Reshetikhin (KR) crystal} $B^{r,s}$, which we will call $B^R$.

\begin{example}
   Let $T=\tableau{1&2\\2&3}$. We have \[f_3(T)=\pr^{-1} \circ f_1 \circ \pr \Big(\,\tableau{1&2\\2&3}\,\Big)=\pr^{-1} \circ f_1 \Big(\,\tableau{1&2\\3&3}\,\Big)=\pr^{-1}\Big(\,\tableau{2&2\\3&3}\,\Big)=\tableau{1&1\\2&2}.\]
   Similarly, $e_3(T)=\tableau{2&2\\3&3}$.
\end{example}

Figure \ref{KR crystal} shows an example of KR crystal graphs.

\begin{figure}[htb] \centering
		\scalebox{0.9}{			
			\begin{tikzpicture}
			\newcommand*{\xdist}{*3}
			\newcommand*{\ydist}{*2.2}
\node (n0) at (0.00\xdist,0\ydist) 
			{$\tableau{1&1\\3&3}$}; 
\node (n1) at (-.8\xdist,1\ydist) 
			{$\tableau{1&1\\2&3}$};
\node (n2) at (.8\xdist,1\ydist) 
			{$\tableau{1&2\\3&3}$};
\node (n3) at (0\xdist,1\ydist) 
			{$\tableau{1&2\\2&3}$};	
\node (n4) at (-1.6\xdist,1\ydist) 
			{$\tableau{1&1\\2&2}$};	
\node (n5) at (1.6\xdist,1\ydist) 
			{$\tableau{2&2\\3&3}$};	
\draw [thick,->] (n0)--(n2) node [left,midway]{$1$};
\draw [thick,->] (n1)--(n0) node [right,midway]{$2$};
\draw [thick,->] (n1)--(n3) node [above,midway]{$1$};
\draw [thick,->] (n3)--(n2) node [above,midway]{$2$};
\draw [thick,->] (n2)--(n5) node [above,midway]{$1$};
\draw [thick,->] (n4)--(n1) node [above,midway]{$2$};
\draw [thick,->,blue] (n3) to[bend right] node [midway,above,inner sep=2pt] {$3$} (n4);
\draw [thick,->, blue] (n5) to[bend right] node [midway,above,inner sep=2pt] {$3$} (n3);
\draw [thick,->,blue] (n2) to[bend left] node [midway,below,inner sep=2pt] {$3$} (n1);
			\end{tikzpicture}	
}			
\caption{The above graph with only black edges is the crystal graph $B_{(2,2)}$ for $n = 3$. Adding in the blue edges, we get the corresponding KR crystal graph $B^{(2,2)}$.} \label{KR crystal}	
\end{figure}

It is easy to see that the set of row (or column) reading words of semistandard tableaux is also closed under $e_n$ and $f_n$.

Let $R=(R_1,\ldots, R_k)$ be a tuple of rectangular partitions, that is, for each $i\in [k]$,  $R_i=(s_i^{r_i})$ for some $s_i, r_i > 0$. We consider the tensor product of KR crystals $B^R=B^{R_1} \otimes \cdots \otimes B^{R_k}$. The vertex set of this graph consists of tuples of SSYT, $T = T_1 \otimes \cdots \otimes T_k$ where $T_i \in SSYT(R_i)$. It is sometimes useful to think of $T$ as a disconnected skew tableau, where the northeast corner of $T_i$ touches the southwest corner of $T_{i+1}$. Unlike their type $A$ analogs (which split into connected components according to the Littlewood-Richardson rule), tensor products of KR crystals are connected.

Following \cite{kang92, shimozono02}, given rectangles $R_1 ,R_2$, there is a unique affine crystal graph isomorphism $\sigma : B^{R_1} \otimes B^{R_2} \to B^{R_2} \otimes B^{R_1}$ called the \emph{combinatorial R-matrix}. It can be easily described in terms of tableaux. Given $T = T_1 \otimes T_2 \in B^{R_1} \otimes B^{R_2}$, $\sigma(T)$ is the unique skew-shape $T_2' \otimes T_1'$ of shape $R_2 \otimes R_1$ obtainable by applying $\jdt$ moves to $T$ (see Example \ref{ex:full_charge_ex} for an example of this computation). For general $\otimes$KR crystals, we let $\sigma_i$ denote the combinatorial $R$-matrix which swaps the $i$ and $i+1$ tensor factors. The $\sigma_i$ satisfy the Yang-Baxter equation, so 
\[
  B^{R_1} \otimes \cdots \otimes B^{R_k} \cong B^{R_1'} \otimes \cdots \otimes B^{R_k'},
\]
where $(R_1',\ldots, R_k')$ is any permutation of $(R_1, \ldots, R_k)$. 

\begin{defn}
Let $R = (R_1, \ldots, R_k)$ be a sequence of rectangles, with $R_i = (s_i^{r_i})$. The $\charge$ function is defined on $B^{R}$ as follows:
\begin{enumerate}[1.]
    \item If $k = 1$, $\charge(T) = 0$ for every $T \in B^{R}$.
    \item If $k = 2$, so $T = T_1 \otimes T_2$, then $\charge(T)$ is equal to the number of cells with row index greater than $\max(r_1, r_2)$ in the tableau $P(\rw(T_1)\rw(T_2))$ obtained by row insertion.
    \item If $k > 2$, then 
    \[
      \charge(T) = \sum_{i < j} \charge(T^{(i)}_{j-1} \otimes T_{j})
    \]
    where $T_{j-1}^{(i)}$ the $(j-1)$'st component of the image of $T$ under the composition of R-matrices $\sigma_{j-2}\cdots \sigma_i$.
\end{enumerate}
\end{defn}
This definition differs slightly from that of \cite{shimozono02}, due to our usage of row insertion instead of column insertion, but it is equivalent. 

\begin{example} \label{ex:full_charge_ex}

Let $R = ((3^3), (2^2),(2^2))$, and consider 

\[
T = \tableau{4&5&6\\ 9 & 13 & 14 \\ 10 & 16 & 17} \otimes \tableau{2 & 8 \\ 3 & 11} \otimes \tableau{1 & 7 \\ 12 & 15}
\]

To compute $\charge(T)$, we must compute $\charge(T_1 \otimes T_2)$, $\charge(T_2 \otimes T_3)$, and \\$\charge(T_2^{(1)} \otimes T_3)$, each of which can be computed by row insertion. The first two are straightforward:
\begin{align*}
    T_1 \leftarrow \rw(T_2) = \ytableaushort{{2}{5}{6}{8}, 
    {3}{11}{14},
    {4}{13}{17},
    {*(blue!30) 9}{*(blue!30)16},
    {*(blue!30) 10}}
    \qquad T_2 \leftarrow \rw(T_3) = \ytableaushort{{1}{7}{12}{15},
    {2}{8},
    {*(blue!30) 3}{*(blue!30)11}}
\end{align*}
To compute the last term, we first compute $T_2^{(1)}$ via the combinatorial $R$-matrix $\sigma_1$:
\[
T_1 \otimes T_2 =
    \tableau{&&&2&8\\
    &&&3&11\\
    4&5&6\\9&13&14\\10&16&17}\quad \xrightarrow{\jdt}\quad
    \tableau{
        &&2&6&8\\
        &&3&11&14\\
        &&5&13&17\\
        4&9\\
        10&16
    } = T_1^{(1)} \otimes T_2^{(1)}
\]
Finally, we see that
\[
  T_2^{(1)} \leftarrow \rw(T_3) = \ytableaushort{{1}{6}{7}{12}{15},
  {2}{8}{14}, {3}{11}{17}, {*(blue!30)5}{*(blue!30)13}}
\]
Thus, $\charge(T) = 3 + 2 + 2 = 7$, where cells that contribue a $+1$ to charge are highlighted in blue.
\end{example}

\begin{prop}[\cite{shimozono02}] \label{p:e_ncharge}
Let $T_1\otimes T_2 \in B^R \otimes B^R$. Then
\[
\charge(e_n(T_1 \otimes T_2)) = \charge(T_1 \otimes T_2) + \begin{cases}
    \hphantom{-}1 & e_n(T_1 \otimes T_2) = e_n(T_1) \otimes T_2 \\
    -1& e_n(T_1\otimes T_2) = T_1 \otimes e_n(T_2)
\end{cases}  
\]
and 
\[
\charge(f_n(T_1 \otimes T_2)) = \charge(T_1 \otimes T_2) + \begin{cases}
    -1 & f_n(T_1 \otimes T_2) = f_n(T_1) \otimes T_2 \\
    \hphantom{-}1 & f_n(T_1\otimes T_2) = T_1 \otimes f_n(T_2)
\end{cases}  
\]
\end{prop}
\begin{example}
To illustrate the above proposition, we consider $T = \tableau{2&3\\7&8} \otimes \tableau{1&5\\4&6}$.
Note that $\charge(T) = 2$, but $f_8(T) = \tableau{1&3\\2&7} \otimes \tableau{1&5\\4&6}$, and $\charge(f_8(T)) = 1$. 
\end{example}

\subsection{Dual Equivalence Graphs}
Dual equivalence graphs were introduced by Assaf in \cite{assaf15} as a general framework to prove that certain sums of quasi-symmetric functions were symmetric and Schur positive. 
We summarize some results here that will be useful to us. However, we follow the slightly altered conventions of \cite[Section 3]{chmutov15} as they match our notation better.
Fix $n > 0$. A signed color graph is a graph $G = (V,E)$ along with the extra data
\begin{itemize}
    \item For each vertex $x \in V$, a set of descents $D(x) \subset [n-1]$.
    \item For each edge $xy \in E$, a set of colors $\tau(E) \subset [n-2]$.
\end{itemize}
Let $E_i = \{e \in E \mid i \in \tau(e)\}$.
A dual equivalence graph is a signed color graph $G = (V,E ,D, \tau)$,  satisfying the following axioms.
\begin{enumerate}
    \item Suppose that $i \in D(x)$ and $i +1 \notin D(x)$. Then there is a unique $y \in V$ such that $xy \in E_i$, $i \notin D(y)$, $i+1 \in D(y)$. 
    \item If $xy \in E_i$, then $i \in D(x)$ iff $i \notin D(y)$, $i+1 \in D(x)$ iff $i + 1  \notin D(y)$  and $j \in D(x)$ iff $j \in D(y)$ for any $j < i-1$ or $j > i +2$. In other words, walking along an edge in $E_i$ swaps $i$ and $i+1$, and does not affect any labels besides $i-1, i, i+1,$ and $i+2$.
    \item Suppose $xy \in E_i$. If $i -1 \in D(x) \Delta D(y)$ (where $\Delta$ is the symmetric difference), then $i-1 \in D(x)$ iff $i+1 \in D(x)$. Similarly, if $i+2 \in D(x) \Delta D(y)$, then $i +2 \in D(x)$ iff $i \in D(x)$.
    \item If $i < n -2$, and we consider the subgraph of $G$ consisting of only $i$ and $i+1$ colored edges, then every connected component has one of the following three forms:
    \begin{center}
    \scalebox{.7}{
    \begin{tikzpicture}
        \node[shape=circle, scale=0.5, fill=black] (n0) at (0,0) {};

        \node[shape=circle, scale=0.5, fill=black] (n1) at (3,0) {};
        \node[shape=circle, scale=0.5, fill=black] (n2) at (5,0) {};
        \node[shape=circle, scale=0.5, fill=black] (n3) at (7,0) {};
        \draw (n1) -- (n2) node[above, midway] {$i$};
        \draw (n2) -- (n3) node[above, midway] {$i+1$};

        \node[shape=circle, scale=0.5, fill=black] (n3) at (10,0) {};
        \node[shape=circle, scale=0.5, fill=black] (n4) at (12,0) {};
        \draw (n3) -- (n4) node[above, midway] {$i, i+1$};
    \end{tikzpicture}
    }
    \end{center}
    If $i < n-3$, then every connected component of the subgraph of $G$ consisting of $i, i+1$, and $i+2$ colored edges has one of the following four forms:
    \begin{center}
    \scalebox{.7}
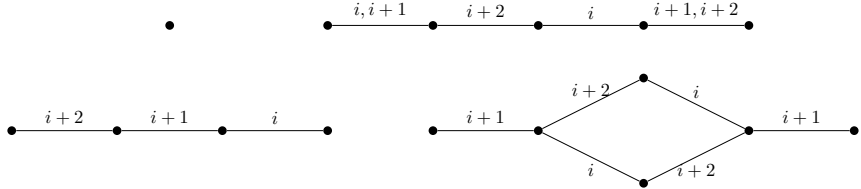
{   
    \begin{tikzpicture}
        \node[shape=circle, scale=0.5, fill=black] (a0) at (0,0) {};

        \node[shape=circle, scale=0.5, fill=black] (b0) at (3,0) {};
        \node[shape=circle, scale=0.5, fill=black] (b1) at (5,0) {};
        \node[shape=circle, scale=0.5, fill=black] (b2) at (7,0) {};
        \node[shape=circle, scale=0.5, fill=black] (b3) at (9,0) {};
        \node[shape=circle, scale=0.5, fill=black] (b4) at (11,0) {};
        \draw (b0) -- (b1) node[above, midway] {$i, i+1$}; 
        \draw (b1) -- (b2) node[above, midway] {$i+2$}; 
        \draw (b2) -- (b3) node[above, midway] {$i$}; 
        \draw (b3) -- (b4) node[above, midway] {$i+1,i+2$}; 

        \node[shape=circle, scale=0.5, fill=black] (c0) at (-3,-2) {};
        \node[shape=circle, scale=0.5, fill=black] (c1) at (-1,-2) {};
        \node[shape=circle, scale=0.5, fill=black] (c2) at (1,-2) {};
        \node[shape=circle, scale=0.5, fill=black] (c3) at (3,-2) {};
        \draw (c0) -- (c1) node[above, midway] {$i+2$};
        \draw (c1) -- (c2) node[above, midway] {$i+1$};
        \draw (c2) -- (c3) node[above, midway] {$i$};

        \node[shape=circle, scale=0.5, fill=black] (d0) at (5,-2) {};
        \node[shape=circle, scale=0.5, fill=black] (d1) at (7,-2) {};
        \node[shape=circle, scale=0.5, fill=black] (d2) at (9,-1) {};
        \node[shape=circle, scale=0.5, fill=black] (d3) at (9,-3) {};
        \node[shape=circle, scale=0.5, fill=black] (d4) at (11,-2) {};
        \node[shape=circle, scale=0.5, fill=black] (d5) at (13,-2) {};
        \draw (d0) -- (d1) node[above, midway] {$i+1$};
        \draw (d1) -- (d2) node[above, midway] {$i+2$};
        \draw (d1) -- (d3) node[below, midway] {$i$};
        \draw (d2) -- (d4) node[above, midway] {$i$};
        \draw (d3) -- (d4) node[below, midway] {$i+2$};
        \draw (d4) -- (d5) node[above, midway] {$i+1$};
    \end{tikzpicture}
    }
    \end{center}
    \item Suppose $xy \in E_i$ and $yz \in E_j$, with $|i - j| \geq 3$. Then there is some $w$ such that $xw \in E_j$ and $wz \in E_i$.
    \item Consider a connected component of the subgraph of $G$ obtained by restricting to edges with colors $\leq i$. Erasing all $i$ colored edges breaks it down into several components, \emph{any two of which} are connected by an edge in $E_i$.
\end{enumerate}

Given a partition $\lambda$, we can construct a dual equivalence graph $G_\lambda$ whose vertices are standard Young tableaux of shape $\lambda$, where there is an edge between $T$ and $T'$ if we can obtain $T'$ by swapping $i$ and $i+1$ in $T$, with the added condition that the descent sets of $T$ and $T'$ are incomparable with respect to containment. Edges in this graph correspond to dual-Knuth moves on the row-reading words of the tableaux (so, in particular, $G_\lambda$ is connected). It turns out that the $G_\lambda$ are essentially all of the dual equivalence graphs.

\begin{thm}[{\cite[Theorem 3.7]{assaf15}}]
Suppose $G$ is a dual equivalence graph, then every connected component of $G$ is isomorphic to $G_\lambda$ for some $\lambda$.
\end{thm}

We will use the following observation of Assaf in the type $A$ case to generalize the notion of a dual equivalence graph to affine type $A$.

\begin{thm}[\cite{assaf08crystals}]
The graph $G_\lambda$ can be constructed by taking the $0$-weight space of the type $A$ crystal graph $B_\lambda$, and connecting vertices by commutators of consecutive crystal operators. Explicitly, $T$ and $T'$ are connected by an edge in $G_\lambda$ if and only if $T' = e_{i}e_{i+1}f_{i}f_{i+1}(T)$ or $T' = e_{i+1}e_if_{i+1}f_i(T)$ for some $i \in [n-1]$.
\end{thm}

\begin{example}\label{ex:deg}
The following is $G_{(3,2)}$, where the descent set of each tableau is highlighted in red.
\begin{center}
\scalebox{.8}{
\begin{tikzpicture}
    \node (n0) at (0,0) {$\ytableaushort{12{\color{red}3},45}$};
    \node (n1) at (3,0) {$\ytableaushort{1{\color{red}2}{\color{red}4},35}$};
    \node (n2) at (6,0) {$\ytableaushort{{\color{red}1}3{\color{red}4},25}$};
    \node (n3) at (9,0) {$\ytableaushort{{\color{red}1}{\color{red}3}5,24}$};
    \node (n4) at (12,0) {$\ytableaushort{1{\color{red}2}5,34}$};
    \draw (n0) -- (n1) node [above, midway]{$23$};
    \draw (n1) -- (n2) node [above, midway]{$1$};
    \draw (n2) -- (n3) node [above, midway]{$3$};
    \draw (n3) -- (n4) node [above, midway]{$12$};
\end{tikzpicture}}
\end{center}
The left-most edge can be computed using the following sequence of crystal operators on their row words.
\[
    45123 \xrightarrow{f_4} 55123 \xrightarrow{f_3} 55124 \xrightarrow{e_4} 45124 \xrightarrow{e_3} 35124
\]

\end{example}

\begin{remark}
In Example \ref{ex:deg} we have added the label $i$ to an edge if exactly one of the adjacent vertices has a descent at $i$ and the other a descent at $i+1$. In later sections, we will instead label edges by $t_i$, where we obtain one tableau from another by swapping $i$ and $i+1$.
\end{remark}

\subsection{Representation of \texorpdfstring{$GL(V)$}{GL(V)}}
We summarize some results in the representation of $GL(V)$ that will be useful in Section \ref{sec:characters}. We refer the reader to \cite{fulton97, stanley_fomin_1999} for a more thorough treatment. Let $V = \CC^n$, and let $\varphi_U : GL(V) \to GL(U)$ be a polynomial representation. The \emph{character} of $\varphi_U$ is a symmetric polynomial given by $\character_U(x_1,\ldots, x_n) = \trace(\varphi_U(\diag(x_1,\ldots x_n)))$, where $\diag(x_1,\ldots, x_n)$ is a diagonal matrix with diagonal entries $x_i$ for $i \in [n]$. If $U$ is a $GL(V)$ module, $U^{\otimes k}$ is a $GL(V) \times S_k$ module, where $GL(V)$ acts diagonally and $S_k$ permutes tensor factors.

Let $C_k$ be a cyclic group (viewed as a subgroup of $S_k$ generated by a $k$-cycle), and let $\chi : C_k \to \CC^{\times}$ be the irreducible character which maps a generator of $C_k$ to a primitive $k$'th root of unity, so that $\chi,\chi^2,\ldots \chi^k$ are all of the irreducible characters of $C_k$. For $m \geq 0$, let 
$e_m = \frac{1}{k}\sum_{\sigma \in C_k} \chi^{-m}(\sigma)\cdot \sigma \in \CC[C_k]$. A quick calculation shows that $e_m$ is idempotent, and that it projects any $C_k$ representation onto its $\chi^m$-isotypic component. The proofs of the following are quick exercises.
\begin{prop}
Viewing $V^{\otimes k}$ as a $GL(V) \times C_k$ module by restriction, we have 
\[
  \character_{e_m V^{\otimes k}}(x_1,\ldots, x_n) = \ell_k^{(m)}
\]
where $\ell_k^{(m)}$ is the Fr\"obenius image of $\Ind_{C_k}^{S_k}(\chi^m)$.
\end{prop}
\begin{cor}
Let $U$ be a polynomial representation of $GL(V)$. Then 
\[
\character_{e_m U^{\otimes k}}(x_1,\ldots, x_n) = \ell^{(m)}_{k} \circ \character_U(x_1,\ldots, x_n)
\]
\end{cor}

\section{KR dual equivalence graphs} \label{sec:krdeg}
Recall that we have $R=(R_1,\ldots, R_k)$ be a sequence of rectangular partitions such that for each $i\in [1,k], R_i=(s_i^{r_i})$ for some $s_i, r_i > 0$. We fix this notation for the rest of the paper. In this section, we introduce a dual equivalence graph defined based on $\otimes$KR crystals $B^R$. We begin by defining a set of vertices $V$ and the descent set for each vertex of $V$ as follows.

\begin{defn}\label{descentset}
   Let $V$ be the set consisting of all weight $0$ elements of $\otimes$KR crystals $B^R=B^{R_1} \otimes \cdots \otimes B^{R_k}$. For an element $T=T_1 \otimes \cdots \otimes T_k\in V$, the \emph{descent set} $D(T)$ of $T$ is defined by 
    \begin{align*}
    D(T) & =\big\{i\in [n-1]: i+1 \,\,\text{appears to the left of}\,\, i\,\, \text{in}\,\, \rw(T)\big\}\\
        &\quad \cup\big\{n: \text{whenever}\,\, 1 \,\,\text{is in }D(\pr(T))\big\}.
    \end{align*}
\end{defn}

\begin{defn}\label{edgeset}
Following the notation from Definition \ref{descentset}, for any $T, T'\in V$ and $i\in \{1,2,\cdots,n\},$ we connect $T$ and $T'$ with an edge if $T'$ can be obtained from $T$ by doing one of the following sequences of moves. 
\begin{enumerate}
    \item For $1\leq i\leq n-1$, we swap $i$ and $i+1$ to obtain $T'$ and $D(T)$ and $D(T')$ are incomparable with respect to inclusion. In this case we write $t_i(T)=T'$.
    \item For $i=n$, if $n\in T_\ell$ and $i+1=1\in T_j$ are not in the same tensor factor, we compute $\jdt_n(T_\ell)$ with replacing the created inner box with $1$, and compute $\jdt_1(T_j)$ with replacing the created outer box with $n$. We obtain $T'=t_n(T)$ if $D(T)$ and $D(T')$ are incomparable.
    \item If $1,2$ and $n$ are in the same tensor factor $T_j$ of $T$ such that $\rw(T_j)=(n\cdots1\,2\cdots)$ and $n\in D(T)$, then we can find $T'$ by the sequence of moves: (i) Apply $\jdt_n(T_j)$ to get an inner box $*$; (ii) Switch $1$ and $2$; (iii) Apply $\jdt_*$ to the tableau obtained from the last move, and replace the resulting outer box with $n$. We denote this edge by $\bar{t}_1(T)=T'$.
    \item If $1,n-1$ and $n$ are in the same tensor factor $T_j$ of $T$ such that $\rw(T_j)=(\cdots n\cdots n-1\cdots 1\cdots)$ and $n\notin D(T)$, then we can find $T'$ by the sequence of moves: (i) Apply $\jdt_n(T_j)$ and fill the inner box with $1$; (ii) Apply $\jdt_{n-1}$ to the tableau obtained from the last move and get an inner box $*$; (iii) Apply $\jdt_1$ for the $1$ in the first row of the tableau followed by applying $\jdt_{*}$, and then fill the two outer boxes with $n-1$ and $n$. We denote this edge by $\bar{t}_{n-1}(T)=T'$.
\end{enumerate}
\end{defn}
We note that edges coming from case (1) correspond to classical dual equivalence moves on the skew tableaux of the shape defined by $R$.

\begin{example} We use the following examples to illustrate the establishment of the edges between $T$ and $T'$, corresponding to the four cases we described above.
    \begin{enumerate}
        \item Let $T=T_1\otimes T_2=\tableau{1&2\\3&5}\otimes \tableau{4&7\\6&8}$ with $D(T)=\{2,4,7,8\}$. For $i=3$, $T'=t_3(T)=\tableau{1&2\\4&5}\otimes \tableau{3&7\\6&8}$ with $D(T')=\{3,7,8\}$. For $i=5,$ $T$ does not have an edge $t_5$ since after switching $5$ and $6$, $D\Big(\,\tableau{1&2\\3&6}\otimes \tableau{4&7\\5&8}\,\Big)=\{2,4,5,7,8\}\supseteq \{2,4,7,8\}=D(T)$. 
        \item For the same $T$ in $(1)$, let $i=8.$ We find $\jdt_8(T_2)$ and replace the created inner box with $1$ to obtain $T'_2=\tableau{1&4\\6&7}$. Also, find $\jdt_1(T_1)$ and replace the created outer box with $8$ gives $T'_1=\tableau{2&5\\3&8}$. Thus $T'=t_8(T)=\tableau{2&5\\3&8}\otimes \tableau{1&4\\6&7}$ with $D(T')=\{1,2,4,7\}.$
        \item Let $T=T_1\otimes T_2=\tableau{1&2&8\\5&10&12\\7&11&13}\otimes \tableau{3&6\\4&9}$ with $D(T)=\{3,4,6,9,10,12,13\}$. $T$ falls into Case (3) of Definition \ref{edgeset}. We take the sequence of moves as \[T_1=\tableau{1&2&8\\5&10&12\\7&11&13}\xrightarrow{\text{$\jdt_{13}$}}\tableau{*&2&8\\1&5&10\\7&11&12}\xrightarrow{\parbox{8mm}{\tiny{switch\\ $1, 2$}}}\tableau{*&1&8\\2&5&10\\7&11&12}\xrightarrow{\text{$\jdt_*$}}\tableau{1&5&8\\2&10&12\\7&11&13},\] and \[\bar{t}_1(T)=T'=T'_1\otimes T'_2=\tableau{1&5&8\\2&10&12\\7&11&13}\otimes \tableau{3&6\\4&9}\] with $D(T')=\{1,3,4,6,9,10,12\}.$

        Now $T'$ falls into Case (4) of Definition \ref{edgeset}. We take the sequence of moves on $T'_1$ as \[T'_1=\tableau{1&5&8\\2&10&12\\7&11&13}\xrightarrow{\text{$\jdt_{13}$}}\tableau{1&1&8\\2&5&10\\7&11&12}\xrightarrow{\text{$\jdt_{12}$}}\tableau{*&1&8\\1&5&10\\2&7&11}\xrightarrow{\text{$\jdt_1$}}\tableau{*&5&8\\1&7&10\\2&11&12}\xrightarrow{\text{$\jdt_*$}}\tableau{1&5&8\\2&7&10\\11&12&13},\] and \[\bar{t}_{12}(T')=T''=T''_1\otimes T''_2=\tableau{1&5&8\\2&7&10\\11&12&13}\otimes \tableau{3&6\\4&9}\] with $D(T'')=\{1,3,4,6,9,10,13\}.$
    \end{enumerate}
\end{example}

Now we are ready to define our dual equivalence graph.

\begin{defn}
Let $R = (R_1, \ldots, R_k)$ be a sequence of rectangles. The \emph{Kirillov-Resehtikhin dual equivalence graph} (KR DEG) of $R$, denoted $\cal{T}(R)$, is the graph with vertex set given by weight $0$ elements of $B^{R_1} \otimes \cdots \otimes B^{R_k}$. Two vertices $T, T'$ are connected by an edge if $g(T) = T'$, for some $g \in \{t_1,\ldots, t_n, \bar{t}_1, \bar{t}_{n-1}\}$.
\end{defn}

Figure \ref{KR DEG} shows an example of the KR DEG $\mathcal{T}((2^2),(1))$.
\begin{figure}[htb] \centering
		\scalebox{.95}{			
			\begin{tikzpicture}
			\newcommand*{\xdist}{*3}
			\newcommand*{\ydist}{*2.2} 
\node (n1) at (1.22\xdist,1.2\ydist) 
			{$\tableau{1&\red{2}\\3&\red{5}}\otimes \tableau{\red{4}}$};
\node (n2) at (-1.22\xdist,1.2\ydist) 
              {$\tableau{\red{1}&\red{3}\\2&4}\otimes \tableau{\red{5}}$};
\node (n3) at (0\xdist,1.2\ydist) 
			{$\tableau{\red{1}&3\\2&5}\otimes \tableau{\red{4}}$};	
\node (n4) at (-2.44\xdist,1.2\ydist) 
			{$\tableau{1&\red{2}\\3&4}\otimes \tableau{\red{5}}$};	
\node (n5) at (2.44\xdist,1.2\ydist) 
			{$\tableau{1&2\\4&\red{5}}\otimes \tableau{\red{3}}$};	
\node (n6) at (1.22\xdist,0\ydist) 
			{$\tableau{1&\red{4}\\3&5}\otimes \tableau{\red{2}}$};
\node (n7) at (-1.22\xdist,0\ydist) 
			{$\tableau{2&\red{3}\\4&5}\otimes \tableau{\red{1}}$};
\node (n8) at (0\xdist,0\ydist) 
			{$\tableau{1&\red{3}\\4&\red{5}}\otimes \tableau{\red{2}}$};	
\node (n9) at (-2.44\xdist,0\ydist) 
			{$\tableau{\red{2}&\red{4}\\3&5}\otimes \tableau{\red{1}}$};	
\node (n10) at (2.44\xdist,0\ydist) 
			{$\tableau{\red{1}&\red{4}\\2&5}\otimes \tableau{\red{3}}$};	
\draw [-{Straight Barb[left]},orange](n2) to [bend right=20, edge label=$\scriptstyle{C_1}$] (n4);
\draw [-{Straight Barb[left]},orange](n4) to [bend left=20, edge label=$\scriptstyle{C_2}$] (n2);
\draw [-] (n2) -- (n4) node [below,midway]{$t_2$};
\draw [-{Straight Barb[left]},orange](n3) to [bend right=20, edge label=$\scriptstyle{C_4}$] (n2);
\draw [-{Straight Barb[left]},orange](n2) to [bend left=20, edge label=$\scriptstyle{C_3}$] (n3);
\draw [-] (n2) -- (n3) node [below,midway]{$t_4$};
\draw [-{Straight Barb[left]},orange](n1) to [bend right=20, edge label=$\scriptstyle{C_5}$] (n3);
\draw [-{Straight Barb[left]},orange](n3) to [bend left=20, edge label=$\scriptstyle{C_1}$] (n1);
\draw [-] (n3)--(n1) node [below,midway]{$t_2$,\bl{$\bar{t}_1$}};
\draw [-{Straight Barb[left]},orange](n5) to [bend right=20, edge label=$\scriptstyle{C_3}$] (n1);
\draw [-{Straight Barb[left]},orange](n1) to [bend left=20, edge label=$\scriptstyle{C_2}$] (n5);
\draw [-] (n1)--(n5) node [below,midway]{$t_3$};
\draw [-{Straight Barb[left]},orange](n7) to [bend right=20, edge label=$\scriptstyle{C_3}$] (n9);
\draw [-{Straight Barb[left]},orange](n9) to [bend left=20, edge label=$\scriptstyle{C_2}$] (n7);
\draw [-] (n9)--(n7) node [below,midway]{$t_3$};
\draw [-{Straight Barb[left]},orange](n8) to [bend right=20, edge label=$\scriptstyle{C_5}$] (n7);
\draw [-{Straight Barb[left]},orange](n7) to [bend left=20, edge label=$\scriptstyle{C_1}$] (n8);
\draw [-] (n7)--(n8) node [below,midway]{$t_1$};
\draw [-{Straight Barb[left]},orange](n6) to [bend right=20, edge label=$\scriptstyle{C_4}$] (n8);
\draw [-{Straight Barb[left]},orange](n8) to [bend left=20, edge label=$\scriptstyle{C_3}$] (n6);
\draw [-] (n8)--(n6) node [below,midway]{$t_3$,\bl{$\bar{t}_4$}};
\draw [-{Straight Barb[left]},orange](n10) to [bend right=20, edge label=$\scriptstyle{C_1}$] (n6);
\draw [-{Straight Barb[left]},orange](n6) to [bend left=20, edge label=$\scriptstyle{C_2}$] (n10);
\draw [-] (n6)--(n10) node [below,midway]{$t_2$};
\draw [-{Straight Barb[left]},orange](n4) to [bend right=35, edge label=$\scriptstyle{C_5}$] (n9);
\draw [-{Straight Barb[left]},orange](n9) to [bend left=35, edge label=$\scriptstyle{C_4}$] (n4);
\draw [-] (n4)--(n9) node [right,midway]{$t_5$};
\draw [-] (n5)--(n10) node [left,midway]{\bl{$\bar{t}_1$},\bl{$\bar{t}_4$}};
\draw [-{Straight Barb[left]},orange](n10) to [bend right=35, edge label=$\scriptstyle{C_4}$] (n5);
\draw [-{Straight Barb[left]},orange](n5) to [bend left=35, edge label=$\scriptstyle{C_5}$] (n10);

			\end{tikzpicture}	
}			
\caption{The KR DEG $\mathcal{T}((2^2),(1))$. The elements of the descent set of each vertex are shown in red. The edges $t_1$-$t_5$ are shown in black and the special edges $\bar{t}_1,\bar{t}_4$ are shown in blue. The edges equivalently obtained by crystal commutators $C_1$-$C_5$ are shown in orange.}\label{KR DEG}	
\end{figure}

Next, we show the connection between our KR DEGs and affine crystals, as a generalization of the result by Assaf \cite{assaf08crystals} on DEGs in type A.

\begin{thm} \label{t:krdeg_crystal}
For any $i\in [n]$, let $C_i$ be the crystal commutator consisting of the composition of the consecutive crystal operators $e_{i}e_{i+1}f_{i}f_{i+1}$. The KR DEG $\mathcal{T}(R)$ can be equivalently generated by its vertex set $V$ and the crystal commutators $C_i$. In other words, for any two vertices $T$ and $T'$ in $V$, $T$ and $T'$ are connected by an edge of $\mathcal{T}(R)$ if and only if $\rw(T)\neq \rw(T')$ and the two reading words are connected by a crystal commutator $C_i$ for some $i$.
\end{thm}

\begin{proof}
Suppose that $T$ and $T'$ are two vertices of $V$ connected by an edge ($t_i$, $\bar{t}_1$, or $\bar{t}_{n-1}$) in $\cal{T}(R)$. Note that $T'$ is obtained from $T$ by swapping $i$ and $i+1$ or conducting certain $\jdt$ moves. In either of the cases specified in Definition \ref{edgeset}, we can assume that, without loss of generality, $i\notin D(T)$ and $i\in D(T')$. Since $D(T)$ and $D(T')$ are not comparable with respect to inclusion, we have either $i-1\in D(T), i-1\notin D(T')$ or $i+1\in D(T), i+1\notin D(T')$. It suffices to consider the first case, since the other follows immediately after reindexing and switching the roles of $T$ and $T'$.
\begin{enumerate}
       \item If $2\leq i\leq n-1$, then the relative positions of $i-1, i, i+1$ in the reading words will be $\rw(T)=(\cdots i\cdots i-1\cdots i+1\cdots)$ and $\rw(T')=(\cdots i+1\cdots i-1\cdots i\cdots)$. We find $C_{i-1}(T)$ as
\begin{align*}
 &e_{i-1}e_{i}f_{i-1}f_{i}(\cdots i\cdots i-1\cdots i+1\cdots)\\ 
            &=e_{i-1}e_{i}f_{i-1}(\cdots i+1\cdots i-1\cdots i+1\cdots)\\
            &=e_{i-1}e_{i}(\cdots i+1\cdots i\cdots i+1\cdots) \\
            &=e_{i-1}(\cdots i+1\cdots i\cdots i\cdots)\\
            &=(\cdots i+1\cdots i-1\cdots i\cdots).
\end{align*}  
We obtain exactly $\rw(T')$ with all other letters fixed, and $i, i+1$ are swapped. Thus $T$ and $T'$ are connected by the crystal commutator $C_{i-1}=e_{i-1}e_{i}f_{i-1}f_{i}$.  
        \item If $i=1$, then $i -1 = n$, so we have $1\notin D(T), n\in D(T)$ and $1\in D(T'), n\notin D(T')$. By definition, this means that $1\in D(\pr(T))$ and $1\notin D(\pr(T'))$. We discuss four sub-cases below. The sub-cases (b1)-(b3) correspond to the usual edge $t_1$ where we simply swap $1$ and $2$ for $T$ and $T'.$ The sub-case (b4) corresponds to the edge $\bar{t}_1$.
        
        $\bf{(2a)}$ If $1$ and $n$ are in distinct tensor factors and $n$ appears to the left of $2$ in the reading word of $T$. It follows that $\rw(T)=(\cdots 1\cdots n\cdots 2\cdots)$ and $\rw(T')=(\cdots 2\cdots n\cdots 1\cdots)$. After promotion, we have $\rw(\pr(T'))=(\cdots 3\cdots 1\cdots 2\cdots)$ since $1\notin D(\pr(T'))$. Further, swapping $1$ and $2$ will not affect the path of $\jdt_n$ so that $\rw(\pr(T))=(\cdots 2\cdots 1\cdots 3\cdots)$. We use different colors on the reading words to represent distinct tensor factors. Note that the computation would be identical regardless of whether $2$ and $n$ are in the same tensor factor. We claim that $\rw(T)$ and $\rw(T')$ are connected by the crystal commutator $C_n=e_{n}e_{1}f_{n}f_{1}$, since 
        \begin{align*}
            &e_{n}e_{1}f_{n}f_{1}(\cdots \gr{1}\cdots \bl{n}\cdots 2\cdots) \\
            &=e_{n}e_{1}(\pr^{-1}f_{1}\pr)(\cdots \gr{2}\cdots \bl{n}\cdots 2\cdots)\\
            &=e_{n}e_{1}(\pr^{-1}f_{1})(\cdots \gr{3}\cdots \bl{1}\cdots 3\cdots)\\
            &=e_{n}e_{1}(\cdots \gr{2}\cdots \bl{1}\cdots 2\cdots) \\
            &=\pr^{-1}e_{1}\pr(\cdots \gr{2}\cdots \bl{1}\cdots 1\cdots)\\
            &=\pr^{-1}(\cdots \gr{3}\cdots \bl{1}\cdots 2\cdots)\\
            &=(\cdots \gr{2}\cdots \bl{n}\cdots 1\cdots).
        \end{align*}

        We make a clarification here regarding the above computation, which also applies to all cases discussed below. Although the promotion operator changes the reading words we choose not to show it for clarity, as the inverse promotion will move all these entries back to their original position.

        $\bf{(2b)}$ If $1$ and $n$ are in distinct tensor factors and $n$ appears to the right of $2$ in the reading word of $T$ (i.e., $2$ and $n$ belong to the same tensor factor consisting of a single row), then $\rw(T)=(\cdots 1\cdots 2\cdots n\cdots)$ and $\rw(T')=(\cdots 2\cdots 1\cdots n\cdots)$. Thus we find $C_n(T)$ as
        \begin{align*}
            &e_{n}e_{1}f_{n}f_{1}(\cdots \gr{1}\cdots\bl{2}\cdots \bl{n}\cdots)\\
            &=e_{n}e_{1}(\pr^{-1}f_{1}\pr)(\cdots \gr{2}\cdots\bl{2}\cdots \bl{n}\cdots)\\
            &=\pr^{-1}e_{1}\pr(\cdots \gr{2}\cdots\bl{1\,\,1}\cdots)\\
            &=(\cdots \gr{2}\cdots\bl{1}\cdots \bl{n}\cdots) = \rw(T').
        \end{align*}
        
        $\bf{(2c)}$ If $1$ and $n$ are in the same tensor factor with $2$ being in a different factor, then $\rw(T)=(\cdots n\cdots 1\cdots 2\cdots)$ and $\rw(T')=(\cdots n\cdots 2\cdots 1\cdots)$. Since $1\in D(\pr (T))$, we also have $\rw(\pr(T))=(\cdots 2\cdots 1\cdots 3\cdots)$ and $\rw(\pr(T'))=(\cdots 3\cdots 1\cdots 2\cdots)$. Thus
        \begin{align*}
            &e_{n}e_{1}f_{n}f_{1}(\cdots \gr{n}\cdots \gr{1}\cdots \bl{2}\cdots)\\ 
            &=e_{n}e_{1}(\pr^{-1}f_{1}\pr)(\cdots \gr{n}\cdots \gr{2}\cdots \bl{2}\cdots)\\
            &=\pr^{-1}e_{1}\pr(\cdots \gr{2}\cdots \gr{1}\cdots \bl{1}\cdots)\\
            &=(\cdots \gr{n}\cdots \gr{2}\cdots \bl{1}\cdots) = \rw(T').
        \end{align*}

        $\bf{(2d)}$ If $1,2$ and $n$ are all in the same tensor factor, one can verify that we have exactly the same computation as in $\bf{(b3)}$ with $1$ and $2$ being adjacent in $\rw(T)=(\cdots n\cdots 1\,\,2\cdots)$. Moreover, by applying $C_n(T)$, we essentially conduct the same sequence of moves on $T$ as in Definition \ref{edgeset}(3) to derive $T'$.
        
  \item  If $i=n,$ we have $n\notin D(T), n-1\in D(T)$ and $n\in D(T'), n-1\notin D(T')$. By definition, $1\notin D(\pr(T))$ and $1\in D(\pr(T'))$. Again, We discuss four sub-cases below. The sub-cases (c1)-(c3) correspond to the edge $t_n$ for which we apply $\jdt$ moves described in Definition \ref{edgeset}(2). The sub-case (c4) corresponds to the edge $\bar{t}_{n-1}$.
  
     $\bf{(3a)}$ If $n$ and $n-1$ are in distinct tensor factors and $n-1$ appears to the left of $1$ in the reading word of $T$, then $\rw(T)=(\cdots n\cdots n-1\cdots 1\cdots)$ and $\rw(T')=(\cdots 1\cdots n-1\cdots n\cdots)$. Similarly, we have $\rw(\pr(T))=(\cdots 1\cdots n\cdots 2\cdots)$ and $\rw(\pr(T'))=(\cdots 2\cdots n\cdots 1\cdots)$. Note that the computation would be identical regardless of $1$ and $n-1$ belonging to the same or distinct tensor factors. Now $\rw(T)$ and $\rw(T')$ are connected by $C_{n-1}(T)$ as
         \begin{align*}
            &e_{n-1}e_{n}f_{n-1}(\pr^{-1}f_{1}\pr)(\cdots \gr{n}\cdots \bl{n-1}\cdots 1\cdots)\\
            &=e_{n-1}(\pr^{-1}e_{1}\pr)(\cdots \gr{1}\cdots \bl{n}\cdots 1\cdots) \\
            &=e_{n-1}\pr^{-1}(\cdots \gr{2}\cdots \bl{1}\cdots 1\cdots) \\
            &=e_{n-1}(\cdots \gr{1}\cdots \bl{n}\cdots n\cdots)\\
            &=(\cdots \gr{1}\cdots \bl{n-1}\cdots n\cdots) = \rw(T').
        \end{align*}
                
       $\bf{(3b)}$ If $n$ and $n-1$ are in distinct tensor factors and $n-1$ appears to the right of $1$ in $T$ (i.e., $1$ and $n-1$ belong to the same tensor factor consisting of a single row), then we have $\rw(T)=(\cdots n\cdots 1\cdots n-1\cdots)$ and $\rw(T')=(\cdots 1\cdots n-1\,\,n\cdots)$, and  
       \begin{align*}
            &e_{n-1}e_{n}f_{n-1}f_{n}(\cdots \gr{n}\cdots \bl{1}\cdots \bl{n-1}\cdots)\\
            &=e_{n-1}e_{n}(\cdots \gr{1}\cdots \bl{1}\cdots \bl{n}\cdots) \\
            &=e_{n-1}(\cdots \gr{1}\cdots \bl{n\,\,n}\cdots)\\
            &=(\cdots \gr{1}\cdots \bl{n-1\,\,n}\cdots).
        \end{align*}
        
      $\bf{(3c)}$ If $n$ and $n-1$ are in the same tensor factor with $1$ being in a different factor, then $\rw(T)=(\cdots n\cdots n-1\cdots 1\cdots)$ and $\rw(T')=(\cdots n-1\cdots 1\cdots n\cdots)$. After promotion we have $\rw(\pr(T))=(\cdots n\cdots 1\cdots 2\cdots)$ and $\rw(\pr(T'))=(\cdots n\cdots 2\cdots 1\cdots)$. Thus
        \begin{align*} 
            &e_{n-1}e_{n}f_{n-1}f_{n}(\cdots \gr{n}\cdots \gr{n-1}\cdots \bl{1}\cdots)\\
            &=e_{n-1}e_{n}(\cdots \gr{n}\cdots\gr{1}\cdots \bl{1}\cdots) \\
            &=e_{n-1}(\cdots \gr{n}\cdots \gr{1}\cdots \bl{n}\cdots)\\
            &=(\cdots \gr{n-1}\cdots \gr{1}\cdots \bl{n}\cdots).
        \end{align*}
        $\bf{(3d)}$ If $1,n-1$ and $n$ are all in the same tensor factor, one can compare that applying $C_{n-1}(T)$, and conduct the sequence of moves on $T$ as in Definition \ref{edgeset}(4) are essentially the same operations to derive $T'$.
        \end{enumerate}
          
        

Now suppose that $\rw(T)\neq \rw(T')$ and $C_{i-1}(T)=T'$ for some $i\in [n]$. We need to check all six relative positions of $i-1, i$ and $i+1$ in $\rw(T)$. For $2\leq i\leq n-1$, after applying the crystal commutator $C_{i-1}$, the relative positions of $i-1, i$ and $i+1$ remain the same in all possibilities except for two cases: First, $e_{i-1}e_{i}f_{i-1}f_{i}(\cdots i\cdots i-1\cdots i+1\cdots)=(\cdots i+1\cdots i-1\cdots i\cdots)$, in which case we exchange $i$ and $i+1$; Second, $e_{i-1}e_{i}f_{i-1}f_{i}(\cdots i\cdots i+1\cdots i-1\cdots)=(\cdots i-1\cdots i+1\cdots i\cdots)$, in which case we exchange $i$ and $i-1$. For $i=1$ (resp. $i=n$), we use the same approach to check all relative positions of $n,1,2$ (resp. $n-1,n,1$) in $\rw(T)$. A brute-force computation with applying $e_{n}e_{1}f_{n}f_{1}$ (resp. $e_{n-1}e_{n}f_{n-1}f_{n}$) on $\rw(T)$ yields either $\rw(T)$ itself, or $\rw(T')$ for some $T'$ connected with $T$ in $\mathcal{T}(R)$, or some reading word that is not from a valid vertex of $\mathcal{T}(R)$. 
\end{proof}

In Figure \ref{KR DEG}, we illustrate all the equivalent edges obtained by the crystal commutators $C_i$ in orange. 
\begin{prop} \label{p:pr_automorphism}
    The map $\pr$ is a graph automorphism of $\cal{T}(R)$, with the additional property that \\$D(\pr(T)) = \{i + 1 \mid i \in D(T)\}$, where addition is taken in $\ZZ/n\ZZ$.
\end{prop}
\begin{proof}
First, since the crystal operators $e_i$  satisfy $e_{i+1} = \pr^{-1} \circ\ e_i \circ \pr$ (and similarly for $f_i$) \cite{shimozono02}, 
it is clear that the edges of $\cal{T}(R)$ (which are compositions of crystal operators) are preserved under promotion. Second, for $i = n$ the statement about descent sets is true by construction. For other $i$, it is well known that jeu de taquin slides do not change descents, so $i \in D(T)$ if and only if $i+1 \in D(\pr(T))$. 
\end{proof}

\begin{figure}[htb] \centering 
\begin{tikzpicture}
    \node(1) at (0,0) {$\tableau{1&2} \otimes \tableau{\red{3}\\\red{4}}$};
    \node(2) at (2.5,0) {$\tableau{2&3} \otimes \tableau{\red{1}\\\red{4}}$};
    \node(3) at (5,0) {$\tableau{3&4} \otimes \tableau{\red{1}\\\red{2}}$};
    \node(4) at (7.5,0) {$\tableau{1&4} \otimes \tableau{\red{2}\\\red{3}}$};
    \node(5) at (3.75, 2) {$\tableau{1&3} \otimes \tableau{\red{2}\\\red{4}}$};
    \node(6) at (3.75, -2) {$\tableau{2&4} \otimes \tableau{\red{1}\\\red{3}}$};

    \draw (1) -- (5) node[above, midway] {$t_2$};
    \draw (2) -- (5) node[left, midway] {$t_1$};
    \draw (3) -- (5) node[left, midway] {$t_4$};
    \draw (4) -- (5) node[above, midway] {$t_3$};
    \draw (1) -- (6) node[below, midway] {$t_4$};
    \draw (2) -- (6) node[left, midway] {$t_3$};
    \draw (3) -- (6) node[left, midway] {$t_2$};
    \draw (4) -- (6) node[below, midway] {$t_1$};

\end{tikzpicture}
\caption{The KR DEG $\cal{T}((2),(1,1))$. \label{f:krdeg_ex2}}
\end{figure}

\begin{example}
Let $R = ((2), (1,1))$. Then $\cal{T}(R)$ can be seen in Figure \ref{f:krdeg_ex2}. Promotion acts on $\cal{T}(R)$ by cycling the middle row once to the right, and exchanging the top and bottom vertices. One can confirm that the descent sets also get cycled appropriately. 
\end{example}

\begin{remark} \label{KLDEG}
In \cite{chmutovlewispylyavskyy23}, Chmutov, Lewis, and Pylyavskyy described the Knuth equivalence classes of affine permutations via the \emph{Affine Matrix Ball Construction} (AMBC) algorithm\cite{chmutovpylyavskyyyudovina18}, whose image gives a pair of tabloids $P$ and $Q$. The Knuth moves on permutations preserve their $P$ tabloids and determine the Knuth moves on their $Q$ tabloids. A modified charge statistic determines the equivalence classes of these $Q$ tabloids, and the induced graph is called the Kazhdan-Lusztig dual equivalence graph (KL DEG). If $R = (R_1, \ldots, R_k)$ with $R_i = (s_i)$, i.e., $R$ is a sequence of single rows, our KR DEG $\mathcal{T}(R)$ reduces to the KL DEG, whose vertex set consists of all $Q$ tabloids of shape $(s_k,\cdots,s_1)$.
\end{remark}

\section{Connected components of KR DEGs} \label{sec:krdegccs}

In this section, we determine the number of connected components of the KR DEG $\cal{T}(R)$ (Theorem \ref{t:connectedcomponents}). This result can be seen as a generalization of \cite[Theorem 8.6]{chmutovpylyavskyyyudovina18} which characterizes the connected components of KL DEGs. 

\begin{thm} \label{t:connectedcomponents}
Let $R = (R_1, \ldots, R_k)$ be a sequence of rectangular partitions. Suppose $R_i$ appears in $R$ with multiplicity $m_i$, and let $d_R = \gcd(m_1,\ldots, m_k)$. Then $\cal{T}(R)$ has $d_R$ connected components, each consisting of all vertices $T$ with a given charge modulo $d_R$.
\end{thm}

\begin{example}
Let $R = ((1^3), (1^3))$. Then $R$ consists of a single shape with multiplicity $2$. Figure \ref{KRDEG2} shows how $\cal{T}(R)$ splits up into two connected components. The left (resp. right) connected component consists of vertices with even (resp. odd) charge.
\end{example}
\begin{figure}[htb] \centering
		\scalebox{0.8}{			
			\begin{tikzpicture}
			\newcommand*{\xdist}{*3}
			\newcommand*{\ydist}{*2.2} 
\node (n1) at (-1.5\xdist,0\ydist) 	
            {$\tableau{1\\3\\5}\otimes \tableau{\red{2}\\\red{4}\\\red{6}}$};
\node (n2) at (-2.5\xdist,1.25\ydist) 			
            {$\tableau{\red{1}\\2\\5}\otimes \tableau{\red{3}\\\red{4}\\\red{6}}$};
\node (n3) at (-1.5\xdist,1.25\ydist) 			
            {$\tableau{1\\\red{3}\\4}\otimes \tableau{\red{2}\\\red{5}\\\red{6}}$};
\node (n4) at (-.5\xdist,1.25\ydist) 			
            {$\tableau{3\\\red{5}\\6}\otimes \tableau{\red{1}\\\red{2}\\\red{4}}$};
\node (n5) at (-2.5\xdist,2.75\ydist) 			
            {$\tableau{\red{1}\\2\\4}\otimes \tableau{\red{3}\\\red{5}\\\red{6}}$};
\node (n6) at (-1.5\xdist,2.75\ydist) 			
            {$\tableau{\red{3}\\4\\6}\otimes \tableau{\red{1}\\\red{2}\\\red{5}}$};
\node (n7) at (-.5\xdist,2.75\ydist) 			
            {$\tableau{2\\\red{5}\\6}\otimes \tableau{\red{1}\\\red{3}\\\red{4}}$};
\node (n8) at (-2.5\xdist,4\ydist) 			               {$\tableau{\red{1}\\\red{2}\\3}\otimes
            \tableau{\red{4}\\\red{5}\\\red{6}}$};
\node (n9) at (-1.5\xdist,4\ydist) 			
            {$\tableau{\red{3}\\\red{4}\\5}\otimes \tableau{\red{1}\\\red{2}\\\red{6}}$};
\node (n10) at (-.5\xdist,4\ydist) 			
            {$\tableau{1\\\red{5}\\\red{6}}\otimes \tableau{\red{2}\\\red{3}\\\red{4}}$};
\draw [-] (n1)--(n2) node [below,midway]{$t_2$};
\draw [-] (n1)--(n3) node [right,midway]{$t_4$};
\draw [-] (n1)--(n4) node [below,midway]{$t_6$};
\draw [-] (n2)--(n5) node [right,midway]{$t_4$};
\draw [-] (n3)--(n6) node [right,midway]{$t_6$};
\draw [-] (n4)--(n7) node [right,midway]{$t_2$};
\draw [-] (n3)--(n5) node [right,midway]{$t_2$};
\draw [-] (n4)--(n6) node [right,midway]{$t_4$};
\draw [-] (n2)--(n7) node [right,at start]{$t_6$};
\draw [-] (n5)--(n8) node [right,midway]{$t_3$};
\draw [-] (n6)--(n9) node [right,midway]{$t_5$};
\draw [-] (n7)--(n10) node [right,midway]{$t_1$};
\quad
\node (n11) at (1.5\xdist,0\ydist) 	
            {$\tableau{2\\4\\6}\otimes \tableau{\red{1}\\\red{3}\\\red{5}}$};
\node (n12) at (.5\xdist,1.25\ydist) 			
            {$\tableau{\red{2}\\3\\6}\otimes \tableau{\red{1}\\\red{4}\\\red{5}}$};
\node (n13) at (1.5\xdist,1.25\ydist) 			
            {$\tableau{2\\\red{4}\\5}\otimes \tableau{\red{1}\\\red{3}\\\red{6}}$};
\node (n14) at (2.5\xdist,1.25\ydist) 			
            {$\tableau{1\\4\\\red{6}}\otimes \tableau{\red{2}\\\red{3}\\\red{5}}$};
\node (n15) at (.5\xdist,2.75\ydist) 			
            {$\tableau{\red{2}\\3\\5}\otimes \tableau{\red{1}\\\red{4}\\\red{6}}$};
\node (n16) at (1.5\xdist,2.75\ydist) 			
            {$\tableau{1\\\red{4}\\5}\otimes \tableau{\red{2}\\\red{3}\\\red{6}}$};
\node (n17) at (2.5\xdist,2.75\ydist) 			
            {$\tableau{1\\3\\\red{6}}\otimes \tableau{\red{2}\\\red{4}\\\red{5}}$};
\node (n18) at (.5\xdist,4\ydist) 			
            {$\tableau{\red{2}\\\red{3}\\4}\otimes \tableau{\red{1}\\\red{5}\\\red{6}}$};
\node (n19) at (1.5\xdist,4\ydist) 			
            {$\tableau{\red{4}\\\red{5}\\6}\otimes \tableau{\red{1}\\\red{2}\\\red{3}}$};
\node (n20) at (2.5\xdist,4\ydist) 			
            {$\tableau{\red{1}\\2\\\red{6}}\otimes \tableau{\red{3}\\\red{4}\\\red{5}}$};
\draw [-] (n11)--(n12) node [below,midway]{$t_3$};
\draw [-] (n11)--(n13) node [right,midway]{$t_5$};
\draw [-] (n11)--(n14) node [below,midway]{$t_1$};
\draw [-] (n12)--(n15) node [right,midway]{$t_5$};
\draw [-] (n13)--(n16) node [right,midway]{$t_1$};
\draw [-] (n14)--(n17) node [right,midway]{$t_3$};
\draw [-] (n13)--(n15) node [right,midway]{$t_3$};
\draw [-] (n14)--(n16) node [right,midway]{$t_5$};
\draw [-] (n12)--(n17) node [right,at start]{$t_1$};
\draw [-] (n15)--(n18) node [right,midway]{$t_4$};
\draw [-] (n16)--(n19) node [right,midway]{$t_6$};
\draw [-] (n17)--(n20) node [right,midway]{$t_2$};

			\end{tikzpicture}	
}			
\caption{The KR DEG $\mathcal{T}((1^3),(1^3))$ has two connected components. The elements of the descent set of each vertex are shown in red. The elements of the left and right components have charge congruent to $0$ and $1 \bmod 2$ respectively.}\label{KRDEG2}	
\end{figure}
It suffices to prove Theorem \ref{t:connectedcomponents} in the special case where $R = (R_1^{\otimes m_1}, \ldots, R_\ell^{\otimes m_\ell})$. Any $R$ where partitions of the same shape are not grouped together can be permuted to this form using combinatorial $R$-matrices without changing the structure of $\cal{T}(R)$.
The remainder of this section is split into two parts. One showing that $\cal{T}(R)$ has at most $d_R$ connected components, and one showing that it has at least $d_R$ connected components. To prove Theorem \ref{t:connectedcomponents}, we will define a modified version of the charge statistic (called \emph{semicharge} and denoted $\scharge$) which agrees with charge modulo $d_R$ in this special case, and use it to prove the following statements:
\begin{enumerate}[1.]
    \item For any $T, T'\in \mathcal{T}(R)$, if $T$ and $T'$ are connected by an edge, then $\scharge(T) \equiv \scharge(T') \bmod d_R$ (Proposition \ref{p:charge_on_edges}).
    \item For any $0 \leq i < d_{R}$, there exists $T\in \mathcal{T}(R)$ such that $\scharge(T) \equiv i \bmod d_{R}$ (Proposition \ref{p:pr_semicharge}).
    \item For any $T, T'\in \mathcal{T}(R)$ such that $\scharge(T) \equiv \scharge(T') \bmod d_R$, $T$ and $T'$ are connected in $\cal{T}(R)$ (Proposition \ref{monster} and Proposition \ref{p:rcomb}).
\end{enumerate}

\subsection{The lower bound of \texorpdfstring{$d_R$}{dR}} \label{subsec:lowerbound}
We start by defining the semicharge, and use it to prove the first two statements above.
\begin{defn} \label{semicharge}
Let $R = (R_1, \ldots, R_k)$, and $T \in B^{R}$. The \emph{semicharge} of $T$, denoted $\scharge(T)$, is given by the following:
\begin{enumerate}[1.]
  \item If $k = 1$, $\scharge(T) = 0$ for every $T \in B^R$.
  \item If $k = 2$, $\scharge(T) = \begin{cases}
    \charge(T) & R_1 = R_2 \\
    0 & R_1 \neq R_2 
  \end{cases}$.
  \item If $k > 2$, then $\scharge(T) = \sum_{i=1}^{k-1} i \cdot \scharge(T_i \otimes T_{i+1})$.
\end{enumerate}
\end{defn}
\begin{example} \label{ex:semicharge_example}
Let $R = ((3^3), (2^2),(2^2))$, and \[
T = \tableau{4&5&6\\ 9 & 13 & 14 \\ 10 & 16 & 17} \otimes \tableau{2 & 8 \\ 3 & 11} \otimes \tableau{1 & 7 \\ 12 & 15},
\]
as in Example \ref{ex:full_charge_ex}. To compute $\scharge(T)$, we see that $\scharge(T_1 \otimes T_2) = 0$ since $R_1 \neq R_2$, and 
\[
T_2 \leftarrow \rw(T_3) = \ytableaushort{{1}{7}{12}{15},
    {2}{8},
    {*(blue!30) 3}{*(blue!30)11}}
\]
so $\scharge(T) = 1\cdot 0 + 2\cdot 2 = 4$. 
\end{example}

Semicharge is similar to the normal charge statistic, but we ignore any contributions between rectangles of different shapes. This makes it easier to compute, since we only compute the local charge between adjacent tensor factors. A drawback to this simplification is that, unlike charge, semicharge is \emph{not} invariant under the application of combinatorial $R$-matrices to permute tensor factors. For instance, if $R = ((2^2), (3^3), (2^2))$, then $\scharge$ is identically $0$ on $B^R$, because no adjacent factors are the same shape. However, if all rectangles of the same shape are grouped together in $R$ (as we are assuming), then $\scharge$ and $\charge$ agree modulo $d_R$, as we show in the following lemma.

\begin{lemma} \label{l:charge_scharge}
Suppose $R = (R_1^{\otimes m_1},\ldots, R_\ell^{\otimes m_\ell})$, and $d_R = \gcd(m_1,\ldots, m_\ell)$. Then for any $T \in B^R$ we have  $\scharge(T) \equiv \charge(T) \bmod d_R$.

\end{lemma}
\begin{proof}
Let $m_0 = 0$ and $\alpha_i =  m_1 + m_2 + \cdots + m_i$. Since combinatorial $R$-matrices act by the identity on rectangles of the same shape, the charge computation reduces to the following
\begin{align*}
\charge(T)=&\sum_{i<j}m_im_j\cdot\charge(T_{\alpha_i}^{(\alpha_{j-1})}\otimes T_{\alpha_{j-1}+1})\\
&+\sum_{i=0}^{\ell-1}\sum_{j=1}^{m_{i+1}-1}j\cdot \charge(T_{m_i+j}\otimes T_{m_i+j+1}),
\end{align*}
and
    \begin{align*}
\scharge(T)&=\sum_{j=1}^{k-1}j\cdot\scharge(T_j\otimes T_{j+1})\\
&=\sum_{i=0}^{\ell-1}\sum_{j=1}^{m_{i+1}-1}(\alpha_i +j)\cdot \scharge(T_{m_i+j}\otimes T_{m_i+j+1}).
\end{align*}
Since each $m_i$ (and therefore each $\alpha_i$) is a multiple of $d_R$, one can directly compare to see that $\scharge(T) \equiv \charge(T) \bmod d_R$.
\end{proof}

\begin{lemma} \label{l:chargechangeontensor}

Let $R$ be a sequence of rectangles, and $T = T_1 \otimes \cdots \otimes T_k \in B^R$. Then
\[
  \scharge(e_n(T)) = \scharge(T) +1 \bmod d_R,   
\]
and 
\[
  \scharge(f_n(T)) = \scharge(T) -1 \bmod d_R.  
\]
\end{lemma}
\begin{proof}
By definition, we have $e_n(T_1 \otimes \cdots \otimes T_k) = T_1 \otimes \cdots \otimes e_n(T_i) \otimes \cdots \otimes T_k$ for some $i$. It follows that the only local semicharges that will change are at $T_{i-1}\otimes T_i$ and $T_i \otimes T_{i+1}$. Using Definition \ref{semicharge} and Proposition \ref{p:e_ncharge}, we know that 
\[
  \scharge(T_{i-1} \otimes e_n(T_i)) = \begin{cases}
    \scharge(T_{i-1} \otimes T_i) - 1 & R_{i-1} = R_i \\
    0 & R_{i-1} \neq R_i
  \end{cases},  
\]
and 
\[
  \scharge(e_n(T_{i}) \otimes T_{i+1})) = \begin{cases}
    \scharge(T_{i} \otimes T_{i+1}) + 1 & R_i = R_{i+1} \\
    0 & R_i \neq R_{i+1}
  \end{cases}.  
\]
If $R_{i-1} = R_i = R_{i+1}$, then 
\[ \scharge(e_n(T)) = \scharge(T) - (i-1) + i = \scharge(T) + 1.\]
If $R_{i-1} = R_i \neq R_{i+1}$, then $i \equiv 0 \bmod d_R$, and 
\[
    \scharge(e_n(T)) = \scharge(T) - (i-1) \equiv \scharge(T) + 1 \bmod d_R.
\]
If $R_{i-1} \neq R_i = R_{i+1}$, then $i \equiv 1 \bmod d_R$, and 
\[
  \scharge(e_n(T)) = \scharge(T) + i \equiv \scharge(T) + 1 \bmod d_R.  
\]
If $R_{i-1} \neq R_i \neq R_{i+1}$, then $d_R = 1$ so the statement is trivial. Thus in all cases $\scharge(e_n(T)) = \scharge(T) +1 \bmod d_R$ as desired. The proof of the second claim is analogous.
\end{proof}
\begin{prop} \label{p:charge_on_edges}
Suppose that $T, T' \in \cal{T}(R)$ are connected by an edge in $\cal{T}(R).$. Then we have $\scharge(T) \equiv \scharge(T') \bmod{d_R}$.
\end{prop}
\begin{proof}
Without loss of generality, we have $T = e_{i}e_{i+1}f_{i}f_{i+1}(T')$ for some $i \in [n]$. If $i, i+1 \in [n-1]$, then $\charge(T) = \charge(T')$ since classical crystal operators do not change charge. Otherwise, one of $i, i+1$ is $n$, and the net change of $\scharge$ will be $0 \bmod d_R$ by Lemma \ref{l:chargechangeontensor}. In either case, the claim follows.
\end{proof}

\begin{lemma} \label{l:pr_semicharge}
Suppose $T_1, T_2$ are increasing rectangular tableaux of shape $(s^r)$, such that $n$ appears at most once between them. Then
\[
  \scharge(\pr(T_1 \otimes T_2)) = \scharge(T_1 \otimes T_2) + \begin{cases}
    -1 & n \in T_1 \\ 
    \hphantom{-}1 & n \in T_2\\ 
    \hphantom{-}0 & \text{otherwise}
  \end{cases}.
\]
\end{lemma}
\begin{proof}
If $n$ is not in either $T_i$, then promotion will only increase the value of every cell by $1$. 
This will not change the semicharge, since the relative order of any two cells (and thus the shape of the resulting insertion tableau) remains the same. For the remaining two cases, we introduce the following notation. As in Definition \ref{promotion}, for $i \in \{1,2\}$, we denote $T_i'$ the result of applying the first step of promotion to $T_i$, and $T_i''$ the result of applying steps $(1)$--$(3)$. In each case, we consider how each step of promotion affects the insertion tableau.\\
($n \in T_1$): Since $T_1$ is a rectangle shape, $n$ must be in the south-east corner of $T_1$. No cell in $T_2$ is bigger than $n$, so any step of the row insertion algorithm that interacts with $n$ will bump it to a lower row. It follows that removing $n$ from $T_1$ will decrease the number of cells with row index greater than $r$ by $1$. Next, we note that applying jeu de taquin is equivalent to a sequence of Knuth moves, and increasing all values by $1$ doesn't change the relative order of any two cells, so steps $(2)$ and $(3)$ don't affect the semicharge either. Finally, let $u = \cw(T_1'')\cw(T_2'')$ and $v = \cw(\pr(T_1))\cw(\pr(T_2)) = u_1\cdots u_{r-1}1\,u_r\cdots u_{2rs-1}$. Consider any collection $A_1,\ldots, A_r$ of disjoint increasing subsequences of $u$ of maximal size. At least one $A_i$ must start with a letter $u_i$ with $i > r$. It follows that $A_1,\ldots, A_r$ is maximal in $u$ if and only if $A_1,\ldots, \{1\}\cup A_i, \ldots, A_r$ is maximal in $v$. Thus, by Greene's theorem \cite[Theorem A1.1.1]{stanley_fomin_1999}, the cell in $P(v)/P(u)$ must be in one of the first $r$ rows (and so not contribute to semicharge). Hence $\scharge(\pr(T_1 \otimes T_2)) = \scharge(T_1 \otimes T_2) - 1$ in this case. \\
($n \in T_2$): We claim after inserting $\cw(T_2)$ into $T_1$, the cell containing $n$ will be in one of the first $r$ rows of the insertion tableau. Indeed, when $n$ is inserted into $T_1$, it will be appended to the end of the first row. Since $n$ is in the southeast corner of $T_2$, there are only $r-1$ cells that are inserted after $n$ that can bump it down. It follows that step $(1)$ doesn't change semicharge. As in the previous case, steps $(2)$ and $(3)$ do not change semicharge either. Finally, adding $1$ to $T_2''$ will increase semicharge by $1$, since when it is inserted into $T_1$ it will shift the first column down by $1$, leading to an extra cell with row index greater than $r$.
\end{proof}

\begin{prop} \label{p:pr_semicharge}
    Let $T\in \cal{T}(R)$ for $R=(R_1,\cdots,R_k)$. Then $\scharge(\pr(T))=\scharge(T)-1 \bmod d_R$.
\end{prop}
\begin{proof}
    We may assume $k>2$ since the case $k=2$ immediately follows from Lemma \ref{l:pr_semicharge}. Since $T=T_1\otimes\cdots\otimes T_k$ has the standard filling, then $n\in T_i$ for some unique $i$. By Lemma \ref{l:pr_semicharge}, promotion will only change the local semicharge of $T_{i-1}\otimes T_i$ and $T_i\otimes T_{i+1}$. That is,
    \[
  \scharge(\pr(T_{i-1} \otimes T_i)) = \begin{cases}
    \scharge(T_{i-1} \otimes T_i) + 1 & R_{i-1} = R_i \\
    0 & R_{i-1} \neq R_i
  \end{cases},  
\]
and 
\[
  \scharge(\pr(T_{i} \otimes T_{i+1})) = \begin{cases}
    \scharge(T_{i} \otimes T_{i+1}) - 1 & R_i = R_{i+1} \\
    0 & R_i \neq R_{i+1}
  \end{cases}.  
\]
Then one can use a similar argument as in Lemma \ref{l:chargechangeontensor} to discuss four cases of shapes $R_{i-1}, R_i$ and $R_{i+1}$. It follows that $\scharge(\pr(T))=\scharge(T)-1 \bmod d_R$.
\end{proof}

\begin{cor} \label{c:cclowerbound}
$\cal{T}(R)$ has at least $d_R$ connected components.
\end{cor}
\begin{proof}
Pick any $T \in \cal{T}(R)$, then $T, \pr(T), \ldots, \pr^{d_R - 1} (T)$ have distinct semicharges modulo $d_R$ by Proposition \ref{p:pr_semicharge}. By Proposition \ref{p:charge_on_edges}, they must belong to different components of $\cal{T}(R)$.
\end{proof}

\subsection{The upper bound of \texorpdfstring{$d_R$}{dR}} \label{subsec:upperbound}
The remainder of this section is spent showing that $\cal{T}(R)$ has \emph{at most} $d_R$ connected components. There are many technical propositions and lemmas that we combine to obtain this result, but we will summarize the general strategy here. We need to show that any two vertices with semicharge equivalent modulo $d_R$ are connected by a path in $\cal{T}(R)$. The main technical challenge comes from the fact that the edges $t_n$, $\bar{t}_1$, and $\bar{t}_{n-1}$ are hard to describe, and so any attempt at describing a path using these edges requires lots of extra bookkeeping. To get around this issue, we do the following two things.
\begin{enumerate}[1.]
\item Find a collection $\cal{A} \subset \cal{T}(R)$ such that each $A_i \in \cal{A}$ is connected to $\pr^{d_i}(A_i)$ for some $d_i$, with the added condition that $\gcd(\{d_i\}) = d_R$.
\item Find a path consisting of either edges in $\cal{T}(R)$ \emph{or} powers of promotion that brings \emph{any} $T \in \cal{T}(R)$ to a fixed $T' \in \cal{T}(R)$.
\end{enumerate}
These two steps, combined with repeated applications of Proposition \ref{p:pr_automorphism} allow us to prove the second half of Theorem \ref{t:connectedcomponents}.

\begin{lemma} \label{partialpr}
    Let $T$ be a rectangular tableau with $r+1$ rows and $s$ columns. Suppose $T$ has consecutive entries in each of its columns such that \[T=
    \ytableausetup{mathmode, boxsize=2em}
    \begin{ytableau}
\scriptstyle{l_1}&\scriptstyle{l_2}&\scriptstyle{l_3}&\cdots&\scriptstyle{l_s} \\ \scriptstyle{l_1+1}&\scriptstyle{l_2+1}&\scriptstyle{l_3+1}&\cdots&\scriptstyle{l_s+1}\\\vdots&\vdots&\vdots&\vdots&\vdots\\\scriptstyle{l_1+r}&\scriptstyle{l_2+r}&\scriptstyle{l_3+r}&\cdots&\scriptstyle{l_s+r}
\end{ytableau}
\]
with $\ell_{i}+r < \ell_{i+1}$.
Then $\pr^{n-(l_s+r)+(r+1)}(T)$ has consecutive entries $\{1,2,\cdots,r+1\}$ in its first column. Furthermore, the rest of its columns are obtained by shifting all entries in column $1$ to column $s-1$ of $T$ to the right by one cell, and adding $n-(l_s+r)+(r+1)$ to each entry.
\end{lemma}

\begin{proof}
    First, note that $\pr^{n-(l_s+r)}(T)$ is obtained by increasing the values of every cell by $n - (l_s + r)$, so the last column has consecutive entries $n-r, n-r+1, \ldots, n$. The next $r+1$ applications of promotion can be computed as follows.
    The first application of $\pr$ will have the effect of shifting the first row to the right (with a one in the northwest corner), and adding one to every cell. This is easy to see, since the jeu de taquin slides will follow the northeast boundary of $T$. Similarly, the $i$'th application of $\pr$ will have the same effect, with the $i$'th row of $T$ getting shifted to the right and $i$ getting placed in the left-most cell of that row. After $r+1$ applications of $\pr$, the net effect is for all columns to get shifted one to the right, with the first column now being $1,\ldots, r+1$. Taking into account the first $n - (l_s + r)$ applications of $\pr$, all elements not in the first column will have increased by $n - (l_s + r) + r + 1$, as desired.
\end{proof}

\begin{remark} \label{partialprr}
    In Lemma \ref{partialpr}, $\pr^{n-(l_s+r)+(r+1)}(T)$ can also be obtained by the following steps:
    \begin{enumerate}
        \item Find $\jdt_{l_s+r}$ and add $1$ to each cell, $r+1$ times.
        \item Add $n-(l_s+r)$ to each entry except those in column $1$.
    \end{enumerate}
\end{remark}

\begin{example}
    Let $n=22, r=2, s=3$, and $T=\tableau{4&8&14\\5&9&15\\6&10&16}$. Then we obtain $\pr^9(T)$ as \[T=\tableau{4&8&14\\5&9&15\\6&10&16}\xrightarrow{\parbox{1.95cm}{\tiny{Shift the first two columns to the right by one cell.}}}\tableau{*&4&8\\*&5&9\\*&6&10}\xrightarrow{\parbox{1.8cm}{\tiny{Fill the first column with 1,2,3.}}}\tableau{\textcolor{red}{1}&4&8\\\textcolor{red}{2}&5&9\\\textcolor{red}{3}&6&10}\xrightarrow{\parbox{1.8cm}{\tiny{Add $9$ to each entry except the first column.}}}\tableau{\textcolor{red}{1}&13&17\\\textcolor{red}{2}&14&18\\\textcolor{red}{3}&15&19}=\pr^9(T).\]
    By Remark \ref{partialprr}, denoting step $(1)$ by $\star$ we can also obtain $\pr^9(T)$ as \[T=\tableau{4&8&14\\5&9&15\\6&10&16}\xrightarrow{\star}\tableau{1&5&9\\6&10&15\\7&11&16}\xrightarrow{\star}\tableau{1&6&10\\2&7&11\\8&12&16}\xrightarrow{\star}\tableau{1&7&11\\2&8&12\\3&9&13}\xrightarrow{\parbox{1.8cm}{\tiny{Add $6$ to each entry except the first column}}}\tableau{1&13&17\\2&14&18\\3&15&19}=\pr^9(T).\]
\end{example}

Now we introduce a variation of the jeu de taquin slide, which can be applied to any cell of a tableau. Here we focus on rectangular partitions for our purposes.

\begin{defn}
 Let $T$ be a rectangular tableau. The \emph{partial jeu de taquin slide} of $T$ in terms of any cell $a$ is defined as follows: 
\begin{enumerate}
    \item Remove all the boxes below and to the right of $a$ so that $a$ becomes an outer box of the rest of the tableau $T'$.
    \item Find $\jdt_a(T')$ using the classical jeu de taquin slide, and add the entry $0$ in the empty spaces on the inner boundary.
    \item Add the removed boxes back to their original positions.
\end{enumerate}
If $a$ is an outer box of $T$, then this partial version is exactly the same as the classical jeu de taquin slide. Using an abuse of notation, we also write it as $\jdt_a(T)$.
\end{defn}

\begin{example}
Let $T=\tableau{1&4&5\\2&6&8\\3&7&9}$, then $\jdt_7(T)=\tableau{0&1&5\\2&4&8\\3&6&9}$ with the last column fixed. Also, $\jdt_6(T)=\tableau{0&1&5\\2&4&8\\3&7&9}$ with all entries below and to the right of $6$ (i.e., the last row and last column) fixed.
\end{example}

\begin{defn}\label{rowcolofT}
    For $R = (R_1, \ldots, R_k)$ with $R_i$ be a rectangular partition of shape $s_i^{r_i}$ for all $i\in [1,k]$, let $T= T_1 \otimes \cdots \otimes T_k\in \cal{T}(R)$. Without loss of generality, we can assume that $s_1\leq s_2\leq \cdots \leq s_k$ and let $r'=\text{max}\{r_i\}$. We define the $i$'th row (resp. column) of $R$ as the union of the cells in the $i$'th row (resp. column) of each $R_j$. If a given rectangle has less than $i$ rows or columns, we ignore it. 
\end{defn}

Next, we define some distinguished elements of $\mathcal{T}(R)$, and prove that they are connected to certain elements in their orbit under promotion. This ends up being helpful in the proof of Theorem \ref{t:connectedcomponents}. 

\begin{defn}\label{mstdcol} 
Let $R = (R_1, \ldots, R_k)$, where each $R_i = (s_i^{r_i})$. Let $r' = \max(r_i)$. For $0 \leq m \leq r'$, we denote by $\row_m(R)$ the element of $\cal{T}(R)$ filled in the following way.
\begin{itemize}
    \item Fill the first $m$ rows of every rectangle from left to right, top to bottom with consecutive entries, starting at $1$.
    \item For the remaining cells, fill the first column of each factor from top to bottom (now reading the factors from right to left), with consecutive entries starting where the last step left off.
    \item Continue filling one column at a time, skipping rectangles if they do not have enough columns.
\end{itemize}
We call $\row_0(R)$ the \emph{column superstandard} filling, and $\row_{r'}(R)$ the \emph{row superstandard} filling, and denote them $\col(R)$ and $\row(R)$ respectively.
\end{defn}
Later we will want to refer to the values of specific cells in $\row_m(R)$. Let $c_{m, i}$ denote the maximum entry in column $i$, and $a_{m,i}$ denote the maximum entry in row $i$. For simplicity, if we set $c_{0,i} = c_i$.

\begin{example}
Let $R = ((2^2),(3^3),(3^3))$. The two special fillings of $T$ in Example \ref{ex22} $(a)$ and $(b)$ give examples for $\col(R)=\row_0(R)$ and $\row_2(R)$, respectively. Also, \[\row(R) = \row_3(R) =\tableau{1&2\\9&10}\otimes \tableau{3&4&5\\11&12&13\\17&18&19}\otimes \tableau{6&7&8\\14&15&16\\20&21&22}.\]
\end{example}

\begin{prop}\label{monster}
Let $T = \row_m(R)$, and fix $d \in \{c_{m,1}, \ldots, c_{m, s_k}, a_{m,m}\}$. Let $S_j$ denote the sequence of edges in $\cal{T}(R)$ given by $S_j := t_{j-(d-1)}t_{j-(d-2)}\cdots t_{j-1}t_j$ for $d\leq j \leq n-1$. Skipping any swaps that are not applicable, we have $S_{n-1}S_{n-2}\cdots S_{d+1}S_d(T) = \pr^{n-d}(T)$. In other words, $T$ and $\pr^{n-d}(T)$ are connected in $\cal{T}(R)$.
\end{prop}

\begin{proof}
We show the case where $m = 0$. Suppose $T = \mathsf{col}(R)$, and suppose $d = c_{i}$. By definition, $d$ is the largest entry in column $i$, and $d+1$ is the smallest entry of column $i+1$. We want to partition the sequence $S_{n-1}S_{n-2}\cdots S_{d+1}S_d$ into intervals in the following way. Suppose that $T_{h_0}, T_{{h_0}+1},\cdots,T_k$ are factors with more than $i$ columns and with number of rows $r_{h_0}, r_{{h_0}+1},\cdots,r_k$. Then partition $S_{d+(\sum_{h_0}^kr_i) -1}\cdots S_{d+1}S_d$ from right to left into intervals with lengths $r_k,r_{k-1},\cdots,r_{h_0}.$ Similarly, for $\ell\in [0,s_k-i-1]$, suppose that $T_{h_\ell}, T_{{h_\ell}+1},\cdots,T_k$ are factors with more than $i+\ell$ columns and with number of rows $r_{h_\ell}, r_{{h_\ell}+1},\cdots,r_k$. Then partition the untouched part of the sequence from right to left into intervals with lengths $r_k,r_{k-1},\cdots,r_{h_\ell}.$ 

Now let $T^{(j)}$ denote the result of applying $S_{d+j-1}\cdots S_d$ to $T$ (so we claim that $T^{(n-d)} = \pr^{n-d}(T)$). Restricting $\rw(T)$ to the entries in $[c_{i-1}+1, c_i+1]$, we get $c_i, c_i - 1, \ldots, c_{i-1} +1, c_i+1$ for each $i$. As a consequence, when we apply $S_d$ to $T$, all swaps will be applicable except for $t_{c_j+1}$, for each $j \in [0, i-1]$. Thus, $T^{(1)}$ can be obtained from $T$ by applying $\jdt_{d+r_k}$ to $T_k$ and increasing all entries in $T$ with value less than $d+r_k$ by one, leaving every other cell untouched. Similarly, for each $j \in [r_k-1]$, we can compute $S_{d+j}(T^{(j)})$ by applying $\jdt_{d+r_k}$ to $T_{k}^{(j)}$ and increasing all entries with value less than $d+r_k$ by one. Using Remark \ref{partialprr}, we can completely characterize $T^{(r_k)}$. 
Column $i+1$ of $T_k$ has been removed as a consequence of continuously applying $\jdt_{d+r_k}$ for $r_k$ times. Column $1$ to column $i$ of $T_k$ gets shifted to the right by one cell. All entries of $T$ with values less than $d+1$ are increased by $r_k$, and every other cell is fixed since none of the transpositions applied have effects on them. Finally, the first column of $T_k^{(r_k)}$ and $T_k^{(n-d)}$ agree by Lemma \ref{partialpr}.

We continue by applying $S_{d+r_k}$ to $T^{(r_k)}$. Similar to the previous sequence of swaps, $T^{(r_k+1)}$ can be obtained from $T^{(r_k)}$ by the followings steps. Apply $\jdt_{d+r_k+r_{k-1}}$ to $T_{k-1}^{(r_k)}$, and increase all entries with values in $(r_k,d+r_k+r_{k-1})$ by one, then add $r_k+1$ to the first entry of the $(k-1)$'st factor, and leave every other cell untouched. Similarly, for each $j \in [r_k,r_k+r_{k-1}-1]$, we can compute $S_{d+j}(T^{(j)})$ by applying $\jdt_{d+r_k+r_{k-1}}$ to $T_{k-1}^{(j)}$ and increasing all entries with value between $r_k$ and $d+r_k+r_{k-1}$ by one, and adding $r_k+1$ to the first entry of the $(k-1)$'st factor. Thus, we can characterize $T^{(r_k+r_{k-1})}$ as follows. Column $i+1$ of $T^{(r_k)}_{k-1}$ has been removed as a consequence of applying $\jdt_{d+r_k+r_{k-1}}$ for $r_{k-1}$ times. Column $1$ to column $i$ of $T^{(r_k)}_{k-1}$ gets shifted to the right by one cell. All entries of $T^{(r_k)}$ with values between $r_k$ and $d+r_k+1$ are increased by $r_{k-1}$, and every other cell is fixed since none of the transpositions applied have effects on them. Finally, the first columns of $T_{k-1}^{(r_k+r_{k-1})}$ and $T_k^{(r_k+r_{k-1})}$ agree with that of $T_{k-1}^{(n-d)}$ and $T_k^{(n-d)}$.

Similar argument follows for applying any intervals of $S_{d+(\sum_{h_0}^kr_i) -1}\cdots S_{d+1}S_d$, and we can completely characterize $T^{\left(\sum^k_{h_0}r_i\right)}$. Column $i+1$ of $T$ has been removed. Column $1$ to column $i$ of $T_{h_0},\cdots, T_k$ gets shifted to the right by one cell. All entries with values less than $d+1$ are increased by $\sum^k_{h_0}r_i$, and every other cell is fixed. Finally, match the first column of $T^{(\sum^k_{h_0}r_i)}$ exactly as $T^{(n-d)}$. The path $T\rightarrow T^{(3)}\rightarrow T^{(6)}\rightarrow T^{(8)}$ in Example \ref{ex22} $(a)$ shows an example when $n=22, d=8, h_0=1, k=3$ and $r_1+r_2+r_3=8$. Similarly, for any $\ell\in [0,s_k-i-1]$, we can obtain $T^{(\sum_{j=h_0}^{h_\ell}\sum_{i=j}^k r_i)}$ as follows. Take $T^{(\sum_{j=h_0}^{h_{\ell-1}}\sum_{i=j}^k r_i)}$, remove column $i+\ell+1$ of it, and for its $h_\ell$-th,$\cdots$, $k$-th factors, shift column $1+\ell$ to column $i+\ell$ to the right by one cell. Then add $\sum^k_{h_\ell}r_i$ to all entries (after column $\ell$) whose values are less than the smallest entry of the removed column $i+\ell+1$. Finally, match the ($\ell+1$)-th column of the $h_\ell$-th,$\cdots$, $k$-th factors exactly as $T^{(n-d)}$. Note that everything before column $1+\ell$ has been matched with $\pr^{n-d}(T)$ and will be fixed throughout the algorithm. We can eventually obtain $\pr^{n-d}(T)$ by inductively applying through all $\ell$. See Example \ref{ex22} $(a)$.

The argument for $m\neq 0$ is similar so we omit the details. Example \ref{ex22} $(b)$ shows a case when $m=2$.
\end{proof}

\begin{remark}
    For $T=\mathsf{row}_m(R)$, the statement in Proposition \ref{monster} is also true for any \\$d\in \{a_{m,1},a_{m,2},\cdots, a_{m,m-1}\}$, and the proof follows similarly. For our purposes, knowing $a_{m,m}$ works is sufficient.
\end{remark}

We illustrate the technique in Proposition \ref{monster} with the following examples.
\begin{example} \label{ex22}
$(a):$ Let $T=T_1\otimes T_2\otimes T_3=\mathsf{col}(R)\in \mathcal{T}((2^2),(3^3),(3^3))$. We choose $d=8$, which is the largest entry in column $1$ of $T$. We consider partitioning part of the sequence $S_{21}S_{20}\cdots S_8$ according to the heights of rectangles with more than one column. That is, $S_{15}S_{14},S_{13}S_{12}S_{11},S_{10}S_9S_8$. Applying them on $T$ from right to left, we obtain $T^{(8)}$ as follows. Remove the second column of $T$ and shift the first column of $T$ to the right by one cell, and also fix the third column. Then add $r_3+r_2+r_1=8$ to every entry with a value less than $d+1=9$. Finally, add back the first column exactly as $T^{(14)}$. To continue, we have intervals of words $S_{21}S_{20}S_{19}, S_{18}S_{17}S_{16}$. Applying them to $T^{(8)}$ from right to left we obtain $T^{(14)}$ as follows. Remove the third column of $T^{(8)}$, and shift the second column of $T^8_2$ and $T^8_3$ to the right by one cell. Then add $r_2+r_3=6$ to every entry (after column $1$) with a value less than $17$ (the smallest entry of column $3$). Finally, fill the empties in the second column exactly as $T^{(14)}$. The path from $T$ to $\pr^{14}(T)$ is shown in Figure \ref{figpr14}.

$(b):$ Let $T=T_1\otimes T_2\otimes T_3=\mathsf{row}_2(R)\in \mathcal{T}(2^2,3^3,3^3)$. We choose $d=a_2=16$. We divide the sequence of words $S_{21}S_{20}\cdots S_{16}$ into intervals according to the heights of rectangles with two rows removed. That is, all heights are $1$. Then we apply $S_{21},S_{20},\cdots S_{16}$ from right to left to obtain $T^{(1)},T^{(2)},\cdots,T^{(6)}=\pr^6(T)$. The path from $T$ to $\pr^6(T)$ is shown in Figure \ref{figpr6}.
\end{example}

\begin{figure}[htb] \centering
\begin{align*}
T=\tableau{7&15\\8&16}\otimes \tableau{4&12&20\\5&13&21\\6&14&22}\otimes \tableau{1&9&17\\2&10&18\\3&11&19}&\xrightarrow{\text{$S_{10}S_9S_8(T)$}}T^{(3)}=\tableau{10&15\\11&16}\otimes \tableau{7&12&20\\8&13&21\\9&14&22}\otimes \tableau{\textcolor{red}{1}&4&17\\\textcolor{red}{2}&5&18\\\textcolor{red}{3}&6&19} \\
&\xrightarrow{\text{$S_{13}S_{12}S_{11}(T^{(3)})$}}T^{(6)}=\tableau{13&15\\14&16}\otimes \tableau{\textcolor{red}{4}&10&20\\\textcolor{red}{5}&11&21\\\textcolor{red}{6}&12&22}\otimes \tableau{\textcolor{red}{1}&7&17\\\textcolor{red}{2}&8&18\\\textcolor{red}{3}&9&19} \\
&\xrightarrow{\text{$S_{15}S_{14}(T^{(6)})$}}T^{(8)}=\tableau{\textcolor{red}{7}&15\\\textcolor{red}{8}&16}\otimes \tableau{\textcolor{red}{4}&12&20\\\textcolor{red}{5}&13&21\\\textcolor{red}{6}&14&22}\otimes \tableau{\textcolor{red}{1}&9&17\\\textcolor{red}{2}&10&18\\\textcolor{red}{3}&11&19} \\
&\xrightarrow{\text{$S_{18}S_{17}S_{16}(T^{(8)})$}}T^{(11)}=\tableau{\textcolor{red}{7}&18\\\textcolor{red}{8}&19}\otimes \tableau{\textcolor{red}{4}&15&20\\\textcolor{red}{5}&16&21\\\textcolor{red}{6}&17&22}\otimes \tableau{\textcolor{red}{1}&\textcolor{red}{9}&12\\\textcolor{red}{2}&\textcolor{red}{10}&13\\\textcolor{red}{3}&\textcolor{red}{11}&14} \\
&\xrightarrow{\text{$S_{21}S_{20}S_{19}(T^{(11)})$}}T^{(14)}=\tableau{\textcolor{red}{7}&21\\\textcolor{red}{8}&22}\otimes \tableau{\textcolor{red}{4}&\textcolor{red}{12}&18\\\textcolor{red}{5}&\textcolor{red}{13}&19\\\textcolor{red}{6}&\textcolor{red}{14}&20}\otimes \tableau{\textcolor{red}{1}&\textcolor{red}{9}&15\\\textcolor{red}{2}&\textcolor{red}{10}&16\\\textcolor{red}{3}&\textcolor{red}{11}&17}
\end{align*}
\caption{The path from $T$ to $\pr^{14}(T)$ for $T=\mathsf{col}(R)\in \mathcal{T}((2^2),(3^3),(3^3)).$}
\label{figpr14}
\end{figure}

\begin{figure}[htb] \centering
\begin{align*}
T=\tableau{1&2\\9&10}\otimes \tableau{3&4&5\\11&12&13\\18&20&22}\otimes \tableau{6&7&8\\14&15&16\\17&19&21}
&\xrightarrow{\text{$S_{16}(T)$}}T^{(1)}=\tableau{2&3\\10&11}\otimes \tableau{4&5&6\\12&13&14\\18&20&22}\otimes \tableau{\textcolor{red}{1}&8&9\\7&16&17\\15&19&21} \\
&\xrightarrow{\text{$S_{17}(T^{(1)})$}}T^{(2)}=\tableau{3&4\\11&12}\otimes \tableau{\textcolor{red}{2}&6&7\\5&14&15\\13&20&22}\otimes \tableau{\textcolor{red}{1}&9&10\\8&17&18\\16&19&21} \\
&\xrightarrow{\text{$S_{18}(T^{(2)})$}}T^{(3)}=\tableau{4&5\\12&13}\otimes \tableau{\textcolor{red}{2}&7&8\\6&15&16\\14&20&22}\otimes \tableau{\textcolor{red}{1}&\textcolor{red}{3}&11\\9&10&19\\17&18&21} \\
&\xrightarrow{\text{$S_{19}(T^{(3)})$}}T^{(4)}=\tableau{5&6\\13&14}\otimes \tableau{\textcolor{red}{2}&\textcolor{red}{4}&9\\7&8&17\\15&16&22}\otimes \tableau{\textcolor{red}{1}&\textcolor{red}{3}&12\\10&11&20\\18&19&21} \\
&\xrightarrow{\text{$S_{20}(T^{(4)})$}}T^{(5)}=\tableau{6&7\\14&15}\otimes \tableau{\textcolor{red}{2}&\textcolor{red}{4}&10\\8&9&18\\16&17&22}\otimes \tableau{\textcolor{red}{1}&\textcolor{red}{3}&\textcolor{red}{5}\\11&12&13\\19&20&21} \\
&\xrightarrow{\text{$S_{21}(T^{(5)})$}}T^{(6)}=\tableau{7&8\\15&16}\otimes \tableau{\textcolor{red}{2}&\textcolor{red}{4}&\textcolor{red}{6}\\9&10&11\\17&18&19}\otimes \tableau{\textcolor{red}{1}&\textcolor{red}{3}&\textcolor{red}{5}\\12&13&14\\20&21&22}
\end{align*}
\caption{The path from $T$ to $\pr^6(T)$ for $T=\mathsf{row}_2(R)\in \mathcal{T}((2^2),(3^3),(3^3)).$}
\label{figpr6}
\end{figure}

\begin{lemma} \label{l:enoughprpowers}
We have that $\gcd\left(\bigcup_m\{c_{m,1},\ldots, c_{m,s_k}, a_{m,m} \}\right) = d_{R}$.
\end{lemma}
\begin{proof}
Let $g_R$ denote the $\gcd$ of the above set. It's clear that $g_R$ is a multiple of $d_R$. Now, note that for $i > 1$, $c_{j,i+1} - c_{j,i}$ equals the number of cells in column $i$ with row index greater than $j$. If $i = 1$, then $c_{j,1} - a_{j,j}$ has the same interpretation. Furthermore, the expression 
\[
\begin{cases}(c_{j+1, i+1} - c_{j+1, i}) - (c_{j, i+1} - c_{j, i}) & i > 1\\
(c_{j+1, 1} - a_{j+1, j+1}) - (c_{j, 1} - a_{j, j}) & i = 1
\end{cases}
\]
gives the number of cells across $R$ with row index $i$ and column index $j$. It's clear that the $\gcd$ of the multiplicities of each \emph{cell} across $R$ must be $d_R$, so it follows that $g_R = d_R$ as desired.
\end{proof}

Now, we show that every $T \in \cal{T}(R)$ can be connected to $\pr^j(\row(R))$ for some $j$. For each $i$, let $a_i$ be the number of cells across all $R_j$ with row index at most $i$ (by convention, let $a_0 = 0$.). Given $T \in \cal{T}(R)$, we want to find a path from $T$ to $\pr^j (\mathsf{row}(R))$ for some $j$. Equivalently, we want to find a sequence of moves that are either (1) a dual equivalence move or (2) a power of promotion, that brings $T$ to $\mathsf{row}(R)$. To this end, we define the \emph{row-combing} algorithm, $\mathsf{rcomb}$. When each $R$ consists of single row partitions, this algorithm essentially reduces to the Shi combing procedure (see \cite{shi91,chmutovpylyavskyyyudovina18}).

First, we define a \emph{partial comb}: Given $T \in \cal{T}(R)$, and $i < j$, such that $D(T) \cap \{i, \ldots, j\} = \{j\}$, let $\mathsf{pcomb}_{i,j}(T)$ be the filling obtained by swapping the descent at $j$ for a descent at $j-1$, which is then swapped for $j-2$, and so on, up to $i$. This results in a new filling $T'$ with $D(T') \cap \{i, \ldots, j\} = \{i\}$. We note that depending on $T$, $\mathsf{pcomb}_{i,j}$ may affect whether or not $i-1$ or $j+1$ are descents as well. We are now ready to define $\mathsf{rcomb}$:
\begin{itemize}
    \item Input: $T \in \cal{T}(R)$.
    \item Output: $T' \in \cal{T}(R)$ that is ``closer'' to $\mathsf{row}(R)$.
    \item Perform the following moves:
    \begin{itemize}
        \item Let $j$ be the minimal element of $D(T) \setminus \{a_i\}$, and suppose $a_i < j < a_{i+1}$.
        \item $T \leftarrow \mathsf{pcomb}_{a_i+1, j}(T)$.
        \item For $\ell = i-1, i-2, \ldots, 0$: $T \leftarrow \mathsf{pcomb}_{a_\ell + 1, a_{\ell+1}}(T)$
        \item $T \leftarrow \pr^{-1}(T)$
    \end{itemize}
    \item Output $T$
\end{itemize}

\begin{example}
We demonstrate one iteration of $\mathsf{rcomb}$ for $R = (3^3, 3^3, 2^2)$. Here, $j = 10$.
\begin{align*}
    T = \tableau{1&2&3\\9&10&15\\11&16&20} \otimes \tableau{4&5&6\\13&18&21\\14&19&22} \otimes \tableau{7&8\\12&17} \xrightarrow{\mathsf{pcomb_{9,10}}}\ &\tableau{1&2&3\\9&11&15\\10&16&20} \otimes \tableau{4&5&6\\13&18&21\\14&19&22} \otimes \tableau{7&8\\12&17} \\
    \xrightarrow{\mathsf{pcomb_{1,8}}}\ &
    \tableau{1&3&4\\2&11&15\\10&16&20} \otimes \tableau{5&6&7\\13&18&21\\14&19&22} \otimes \tableau{8&9\\12&17} \\
    \xrightarrow{\pr^{-1}}\ & \tableau{1&2&3\\9&10&14\\15&19&22}\otimes \tableau{4&5&6\\12&17&20\\13&18&21} \otimes \tableau{7&8\\11&16}
\end{align*}
Note that the entries $1$ through $10$ are already where they should be in $T$, and that the first cell (reading left to right) that disagrees with $\mathsf{row}(R)$ is decreased by one $(15 \to 14)$.
\end{example}

\begin{lemma} \label{l:pcomb}
Let $T \in \cal{T}(R)$, and fix $i \leq j$. Suppose that $D(T) \cap \{i, \ldots, j\} = \{j\}$, and that the cells of $T$ containing $\{i,\ldots, j+1\}$ are contained with at most two rows of the factors of $T$. Let $c$ be the smallest value in $\{i,\ldots, j+1\}$ that follows $j+1$ in $\rw(T)$. Then $\mathsf{pcomb}_{i,j}(T)$ has the affect of cyclically shifting the cells containing $c+1, \ldots, j+1$ up by one (each cell gets increased by $1$, except for $j+1$ which wraps around to $c+1$), and cyclically shifting the cells containing $i,\ldots, c$ up by one.
\end{lemma}
\begin{proof}
This can be seen directly by noting that under the assumptions of the Lemma statement, $\rw(T)$ restricted to the alphabet $\{i, \ldots, j+1\}$ has the form 
\[
  i\ (i+1)\ \cdots\ (c-1)\ (j+1)\ c\ (c+1)\ \cdots\ (j-1)\ j
\]
We can apply dual Knuth moves swapping $j+1$ and $j$, followed by $j$ and $j-1$, and so on, until we swap $c+1$ and $c+2$, at which point, the row word looks like 
\[
  i\ (i+1)\ \cdots\ (c-1)\ (c+1)\ c\ (c+2)\ \cdots\ j\ (j+1)
\]
Since swapping $c, c+1$ is \emph{not} a dual Knuth move for this word, we continue swapping $c-1$ with $c$, $c-2$ with $c-1$, and so on, to get the word 
\[
  (i+1)\ (i+2)\ \cdots\ c\ (c+1)\ i\ (c+2)\ \cdots\ j\ (j+1)
\]
Since elementary dual equivalences on $\rw(T)$ correspond to edges in $\cal{T}(R)$, the claim follows.
\end{proof}
We note that under the assumptions of the above Lemma, if the cells containing $\{i,\ldots, j+1\}$ lie in two rows of $T$, then $j+1$ is the unique cell in the lower row, and $c$ is the cell directly above.
\begin{prop} \label{p:rcomb}
Repeatedly applying $\mathsf{rcomb}$ to any $T \in \cal{T}(R)$ will eventually reach $\mathsf{row}(R)$ after finitely many iterations.
\end{prop}
\begin{proof}
In order to prove that repeated applications of $\mathsf{rcomb}$ eventually lead to $\mathsf{row}(R)$, we will use the following strategy. Consider the cells of $T$ that have values strictly larger than their corresponding value in $\mathsf{row}(R)$, and let $\alpha$ be the one with the smallest value. We claim that $\mathsf{rcomb}(T)$ fixes all entries of $T$ that precede $\alpha$, and decreases the value of $\alpha$ itself. 

First, it is clear that no application of $\mathsf{pcomb}$ will increase the value of $\alpha$. Indeed, the only step of the algorithm that could interact with $\alpha$ is the application of $\mathsf{pcomb}_{a_i + 1, j}$, but this will only happen if the value of $\alpha$ is the $c$ of Lemma \ref{l:pcomb}, or $j+1$. In either case, it will decrease, not increase.

After applying each $\mathsf{pcomb}$, the cycling of entries described in Lemma \ref{l:pcomb} implies that the cells less than $\alpha$ will have increased by $1$, except for a single factor of $T$ (the factor where the initial cell that played the role of $c$ in $\mathsf{pcomb}_{a_i+1, j}$ is located). The first $i$ rows of $T$ will look like

\[
\ytableausetup{boxsize=6ex}
\begin{ytableau}
    \scriptstyle{1}& \scriptstyle{\ell_1 + 2} & \scriptstyle{\ell_1 + 3}& \cdots &  \scriptstyle{\ell_1 + h} \\
    \scriptstyle{\ell_1 + 1} & \scriptstyle{\ell_2 + 2} & \scriptstyle{\ell_2 + 3} &\cdots & \scriptstyle{\ell_2 + h} \\
    \scriptstyle{\ell_2+ 1} & \vdots & \vdots & \vdots  & \vdots \\
    \vdots & \scriptstyle{\ell_{i}+2}& \cdots &\alpha &\\
    \scriptstyle{\ell_{i} +1} & & & &
\end{ytableau}
\]
where $\ell_i$ is the value of the leftmost cell in that row in $\row(R)$. 
Since $1$ is always in this factor of $T$, the effect of $\pr^{-1}$ on the other factors will be to decrease all elements by $1$. For this factor, it is clear that the $\jdt$ path will go down the left column of the tableau up to at least row $i$. This will return all values in this factor that were originally less than $\alpha$ back to what they were. If $\alpha$ happens to be in this factor, then it will be in row $i$ but not column one, so $\pr^{-1}$ will just decrease the value by $1$.
\end{proof}

\begin{example}
Let $R = ((2^2),(2^2))$, then the following is a complete application of the row combing procedure. 
\[T = \tableau{1&2\\3&6} \otimes \tableau{4&5\\7&8}\xrightarrow{\mathsf{rcomb}} \tableau{1&2\\5&8} \otimes \tableau{3&4\\6&7}\xrightarrow{\mathsf{rcomb}} \tableau{1&2\\5&6} \otimes \tableau{3&4\\7&8} = \row(R).\]
\end{example}

\begin{cor}\label{corprj}
    Any vertex $T\in \mathcal{T}(R)$ can be connected to $\pr^j(\mathsf{row}(R))$ for some power $j$.
\end{cor}

Now we are ready to combine all results and prove the main theorem. 
\begin{proof}[Proof of Theorem \ref{t:connectedcomponents}]
We know from Corollary \ref{c:cclowerbound} that $\cal{T}(R)$ has at least $d_R$ connected components. What remains to show is that any two $T,T'$ with equal charge are connected in $\cal{T}(R)$. By Corollary \ref{corprj}, it suffices to show that $\mathsf{row}(R)$ is connected to $\pr^{d_R}\mathsf{row}(R)$. Let $A = \{c_{i,j} \mid i\leq r', j \leq s_k\} \cup \{a_{i,i} \mid i \leq r'\}$. We know by Proposition \ref{monster} that for every $d_i$ in $A$, there is some $T \in \cal{T}(R)$ such that $T$ is connected to $\pr^{n - d_i}(T)$, and since $\pr^n = \text{id}$ on $\cal{T}(R)$, this implies that $T$ is also connected to $\pr^{d_i}(T)$. By Corollary \ref{corprj}, $T$ is connected to $\pr^{k_i}(\row(R))$ for some $k_i$. Since $\pr$ is a graph automorphism, this means that $\pr^{d_i}(T)$ is connected to $\pr^{k_i + d_i}(\row(R))$. Concatenating the paths, we get that $\pr^{k_i}(\row(R))$ is connected to $\pr^{k_i + d_i}(\row(R))$, which implies that $\row(R)$ is connected to $\pr^{d_i}(\row(R))$ for each $d_i$ in $A$. Since $\gcd(A) = d_R$ by Lemma \ref{l:enoughprpowers}, we can conclude that $\row(R)$ is connected to $\pr^{d_R}(\row(R))$, as desired.
\end{proof}

\section{Characters of connected components} \label{sec:characters}
Recall that, given a subset $A \subset [n-1]$, the associated fundamental quasisymmetric function $F_A$ is given by 
\[
F_A(x) = \sum_{\substack{i_1 \leq \cdots \leq i_n\\ a \in A \Rightarrow i_a < i_{a+1}}} x_{i_1}\cdots x_{i_n}.
\]
\begin{thm}
Let $R = (R_1,\ldots, R_k)$. For $T \in \cal{T}(R)$, let $\overline{D}(T) = D(T) \setminus\{n\}$, then
\[
\sum_{T \in \cal{T}(R)} F_{\overline{D}(T)} = s_{R_1}\cdots s_{R_k}
\]
where $F_A$ is the fundamental quasisymmetric function corresponding to $A \subset [n-1].$
\end{thm}
\begin{proof}
Regard $B^{R_1} \otimes \cdots \otimes B^{R_k}$ as a type $A$ crystal. It decomposes into connected components, each of which is isomorphic to a type $A$ crystal $B_{\lambda}$, where the multiplicity of $\lambda$ is exactly the the multiplicity of $s_\lambda$ in $s_{R_1}\cdots s_{R_k}$ \cite[Theorem 9.5]{bumpschilling17}. Next, we know from \cite{assaf08crystals} that if we take the $0$-weight space of the type $A$ crystal $B_{\lambda}$ and regard it as a dual-equivalence graph, then the sum of fundamental quasisymmetric functions indexed by descents is the Schur function $s_\lambda$. The claim follows.
\end{proof}

If we consider the same sum of fundamental quasisymmetric functions restricted to a single connected component of $\cal{T}(R)$, we get a Schur positive symmetric function, since $\cal{T}(R)$ is a dual equivalence graph if we forget about the extra affine edges. Conjecturally, it appears that the symmetric functions that we get come from the plethysm of \emph{cyclic characters} with certain products of rectangular Schur functions. 
Let \[\ell_{k}^{(i)} = \frac{1}{k}\sum_{d\mid k} c(i,d)p_d^{n/d}\]
where $c(i,d) = \tfrac{\mu(i')\varphi(d)}{\varphi(i')}$, where $i' = \tfrac{i}{\gcd(i,d)}$, $\mu$ is the (number-theoretic) M\"obius function, and $\varphi$ is Euler's totient function. We call such $\ell_{k}^{(i)}$ \emph{cyclic characters}, as they were first studied by Foulkes as the Fr\"obenius images of the symmetric group representations obtained by inducing from an irreducible representation of a maximal cyclic subgroup \cite{Foulkes72}.
\begin{conj} \label{conj:characters}
Suppose that $R$ has distinct rectangles $R_j$ that appear with multiplicity $m_j$ (so $d_R = \gcd (m_j)$). Then for $i \in [d_R]$, the charge $i$ connected component of $\cal{T}(R)$ has character given by the plethysm
\[
\ell_{d_R}^{(i)} \left[s_{R_1}^{m_1/d_R}\cdots s_{R_k}^{m_k/d_R}\right].
\]
\end{conj}
We note that since $\sum_{i = 1}^{d_R} \ell_{d_R}^{i} = p_1^{d_R}$, we can use simple plethystic identites to show that the sum of all such characters is the product of Schur functions from the previous theorem. For some special cases, a Schur expansion of these plethysms is known. If every $R_i$ is a column (so $s_{R_i}$ is an elementary symmetric function) or if every $R_i$ is a row (so $S_{R_i}$ is a homogeneous symmetric function), the expansion is known due to Lascoux, Leclerc, and Thibon \cite{LLT94}. The case of all $R_i$ being a fixed rectangle shape is known due to Iijima \cite{Iijima13}.

We can interpret the plethysms appearing in Conjecture \ref{conj:characters} as follows. Let $U$ be a tensor product of $GL_n(\CC)$ representations corresponding to multiples of fundamental weights, so that the character of $U$ is given by a product of rectangular Schur functions. Then the diagonal action of $GL_n(\CC)$ on $U^{\otimes k}$ commutes with the action of the cyclic group $C_k$ which cycles tensor factors. Thus, if $\chi$ is an irreducible character for $C_k$ corresponding to a primitive $k$'th root of unity, the corresponding isotypic components of $U^{\otimes k}$ are representations of $GL_n(\CC)$, and their characters correspond exactly to $\ell_{k}^{(i)}[\character_U(x_1,\ldots, x_n)]$ for some $i$.
 
\begin{example}
    Consider $R = ((1^3), (1^3))$ as in Figure \ref{KRDEG2}. The left and right connected components consist of the semicharge $0$ and $1$ vertices of $\cal{T}(R)$ respectively. One can compute that $\ell_2^{(0)} = s_2$ and $\ell_2^{(1)} = s_{1,1}$, and the characters of each connected component are given by 
    \begin{align*}
    2F_{1245} + F_{1234} + F_{2345} + F_{135}+F_{1235}+ F_{1345} + F_{134} + F_{235} + F_{24} = s_{2, 1, 1, 1, 1}+ s_{2, 2, 2} &= s_2 [s_{1,1,1}], \\
    F_{1235}+F_{12345}+F_{1345} + F_{124}+F_{234}+F_{245} + F_{1245} + F_{134} + F_{235} + F_{135}= s_{1, 1, 1, 1, 1, 1} + s_{2, 2, 1, 1} &=s_{1,1}[s_{1,1,1}]
    \end{align*}
\end{example}  
\begin{example}
Let $R = ((1^2), (1^2), (2),(2))$, so $d_R = 2$. $\cal{T}(R)$ has $2520$ vertices and two connected components. The characters of the two components are exactly $\ell_{2}^{(0)}[s_2s_{1,1}]$ and $\ell_{2}^{(1)}[s_2s_{1,1}]$:
\begin{align*}
    \ell^{(0)}_{2}[s_2s_{1,1}] &= s_{6,2} + 2s_{5,2,1} + 2s_{5,1,1,1} + s_{4,4} + 2s_{4,3,1} + 3s_{4,2,2} + 2s_{4,2,1,1} + 2s_{4,1,1,1,1} + 3s_{3,3,1,1} \\
    &\quad + 2s_{3,2,2,1} + 2s_{3,2,1,1,1} + s_{2,2,2,2} + s_{2,2,1,1,1,1},\\
    \ell^{(1)}_{2}[s_2s_{1,1}] &= s_{6,1,1}+s_{5,3} +2 s_{5,2,1}+ s_{5,1,1,1}+ 2 s_{4,3,1}
    + s_{4,2,2} + 4 s_{4,2,1,1} + s_{4,1,1,1,1}+ 2 s_{3,3,2} \\
    &\quad + s_{3,3,1,1}+ 2 s_{3,2,2,1} + 2 s_{3,2,1,1,1}+ s_{3,1,1,1,1,1} +s_{2,2,2,1,1}.
\end{align*}
\end{example}
\begin{example}
    Let $R = ((2^2), (2^2), (2^2))$, so $d_R = 3$. $\cal{T}(R)$ has $277200$ vertices and the characters of its three connected components are exactly $\ell_{3}^{(i)}[s_{2,2}]$ for each $i$. Since $\ell^{(1)}_3 = \ell_3^{(2)}$, the characters for two of the components coincide.
\end{example}
\bibliographystyle{abbrv}
\bibliography{refs}

@article {chmutovlewispylyavskyy23,
    AUTHOR = {M.~Chmutov and J.~Lewis and P.~Pylyavskyy},
     TITLE = {Monodromy in {K}azhdan-{L}usztig cells in affine type {A}},
   JOURNAL = {Math. Ann.},
    VOLUME = {386},
      YEAR = {2023},
    NUMBER = {3-4},
     PAGES = {1891--1949}
}

@article {shimozono02,
  AUTHOR = {M.~Shimozono},
  TITLE = {Affine type {A} crystal structure on tensor products of rectangles, {D}emazure characters, and nilpotent varieties},
  JOURNAL = {J. Algebraic Combin.},
  VOLUME = {15},
  YEAR = {2002},
  NUMBER = {2},
  PAGES = {151--187},

}

@book {fulton97,
    AUTHOR = {W.~Fulton},
     TITLE = {Young tableaux},
    SERIES = {London Mathematical Society Student Texts},
    VOLUME = {35},
      NOTE = {With applications to representation theory and geometry},
 PUBLISHER = {Cambridge University Press, Cambridge},
      YEAR = {1997},
     PAGES = {x+260},
}

@article{assaf08crystals,
    title      = {A combinatorial realization of {S}chur-{W}eyl duality via crystal graphs and dual equivalence graphs},
    author     = {S. Assaf},
    url        = {https://dmtcs.episciences.org/3626},
    doi        = {10.46298/dmtcs.3626},
    journal    = {Discrete Math. Theor. Comput. Sci.},
    issn       = {1365-8050},
    volume     = {DMTCS Proceedings vol. AJ, 20th Annual International Conference on Formal Power Series and Algebraic Combinatorics (FPSAC 2008)},
    issuetitle = {Proceedings},
    eid        = 33,
    year       = {2008},
    month      = {Jan},
    language   = {English},
}

@article {assaf15,
    AUTHOR = {S.~Assaf},
     TITLE = {Dual equivalence graphs {I}: {A} new paradigm for {S}chur
              positivity},
   JOURNAL = {Forum Math. Sigma},
  FJOURNAL = {Forum of Mathematics. Sigma},
    VOLUME = {3},
      YEAR = {2015},
     PAGES = {Paper No. e12, 33},
      ISSN = {2050-5094},
       DOI = {10.1017/fms.2015.15},
       URL = {https://doi.org/10.1017/fms.2015.15},
}

@article {shi91,
    AUTHOR = {J.~Y.~Shi},
     TITLE = {The generalized {R}obinson-{S}chensted algorithm on the affine
              {W}eyl group of type {$\tilde A_{n-1}$}},
   JOURNAL = {J. Algebra},
  FJOURNAL = {Journal of Algebra},
    VOLUME = {139},
      YEAR = {1991},
    NUMBER = {2},
     PAGES = {364--394},
      ISSN = {0021-8693,1090-266X},
       DOI = {10.1016/0021-8693(91)90300-W},
       URL = {https://doi.org/10.1016/0021-8693(91)90300-W},
}

@article {chmutovpylyavskyyyudovina18,
    AUTHOR = {M.~Chmutov and P.~Pylyavskyy and E.~Yudovina},
     TITLE = {Matrix-ball construction of affine {R}obinson-{S}chensted
              correspondence},
   JOURNAL = {Selecta Math. (N.S.)},
  FJOURNAL = {Selecta Mathematica. New Series},
    VOLUME = {24},
      YEAR = {2018},
    NUMBER = {2},
     PAGES = {667--750},
      ISSN = {1022-1824,1420-9020},
       DOI = {10.1007/s00029-018-0402-6},
       URL = {https://doi.org/10.1007/s00029-018-0402-6},
}

@article {chmutov15,
    AUTHOR = {M.~Chmutov},
     TITLE = {Type {$A$} molecules are {K}azhdan-{L}usztig},
   JOURNAL = {J. Algebraic Combin.},
  FJOURNAL = {Journal of Algebraic Combinatorics. An International Journal},
    VOLUME = {42},
      YEAR = {2015},
    NUMBER = {4},
     PAGES = {1059--1076},
      ISSN = {0925-9899,1572-9192},
       DOI = {10.1007/s10801-015-0616-z},
       URL = {https://doi.org/10.1007/s10801-015-0616-z},
}

@article {assafbilley12,
    AUTHOR = {S.~Assaf and S.~Billey},
     TITLE = {Affine dual equivalence and {$k$}-{S}chur functions},
   JOURNAL = {J. Comb.},
  FJOURNAL = {Journal of Combinatorics},
    VOLUME = {3},
      YEAR = {2012},
    NUMBER = {3},
     PAGES = {343--399},
      ISSN = {2156-3527,2150-959X},
       DOI = {10.4310/JOC.2012.v3.n3.a5},
       URL = {https://doi.org/10.4310/JOC.2012.v3.n3.a5},
}

@book{bumpschilling17,
  title={Crystal Bases: Representations And Combinatorics},
  author={Bump, D. and Schilling, A.},
  isbn={9789814733465},
  url={https://books.google.com/books?id=Ncw5DwAAQBAJ},
  year={2017},
  publisher={World Scientific Publishing Company}
}

@incollection {Foulkes72,
    AUTHOR = {H.~O.~Foulkes},
     TITLE = {Characters of symmetric groups induced by characters of cyclic
              subgroups},
 BOOKTITLE = {Combinatorics ({P}roc. {C}onf. {C}ombinatorial {M}ath.,
              {M}ath. {I}nst., {O}xford, 1972)},
     PAGES = {141--154},
 PUBLISHER = {Inst. Math. Appl., Southend-on-Sea},
      YEAR = {1972},
}

@book{stanley_fomin_1999, place={Cambridge}, series={Cambridge Studies in Advanced Mathematics}, title={Enumerative Combinatorics}, volume={2},  publisher={Cambridge University Press}, author={R.~P.~Stanley}, year={1999}, collection={Cambridge Studies in Advanced Mathematics}}

@article{kang92,
title = {Affine crystals and vertex models},
author = {S.~J.~Kang and M.~Kashiwara and K.~C.~Misra and T.~Miwa and T.~Nakashima and A.~Nakayashiki},
journal = {Int. J. Mod. Phys. A},
number = {1},
pages = {449--484},
volume = {7},
year = {1992}
}

@article{KR1990,
  title={Representations of {Y}angians and multiplicities of occurrence of the irreducible components of the tensor product of representations of simple Lie algebras},
  author={A.~N.~Kirillov and N.~Reshetikhin},
  journal={J. Sov. Math.},
  year={1990},
  volume={52},
  pages={3156-3164},
  url={https://api.semanticscholar.org/CorpusID:121921415}
}

@article{Stembridge08,
author = {J.~R.~Stembridge},
year = {2008},
month = {10},
pages = {346-368},
title = {Admissible ${W}$-graphs},
volume = {12},
journal = {Represent. Theory},
doi = {10.1090/S1088-4165-08-00336-1}
}

@article{Stembridge12,
title = {A finiteness theorem for ${W}$-graphs},
journal = {Adv. Math.},
volume = {229},
number = {4},
pages = {2405-2414},
year = {2012},
issn = {0001-8708},
doi = {https://doi.org/10.1016/j.aim.2012.01.008},
url = {https://www.sciencedirect.com/science/article/pii/S0001870812000217},
author = {J.~R.~Stembridge},
}

@article{naoi13,
title = {{D}emazure crystals and tensor products of perfect {K}irillov–{R}eshetikhin crystals with various levels},
journal = {J. Algebra},
volume = {374},
pages = {1-26},
year = {2013},
issn = {0021-8693},
doi = {https://doi.org/10.1016/j.jalgebra.2012.10.020},
url = {https://www.sciencedirect.com/science/article/pii/S0021869312005315},
author = {K.~Naoi},
}

@article{schillingtingley11,
author = {A.~Schilling and P.~Tingley},
year = {2011},
month = {04},
pages = {},
title = {{D}emazure Crystals, {K}rillov-{R}eshetikhin Crystals, and the Energy Function},
volume = {19},
journal = {Electron. J. Combin.},
doi = {10.37236/2184}
}

@article{lnsss14,
    author = {C.~Lenart and S.~Naito and D.~Sagaki and A.~Schilling and M.~Shimozono},
    title = {A Uniform Model for {K}irillov–{R}eshetikhin Crystals {I}: Lifting the Parabolic Quantum {B}ruhat Graph},
    journal = {Int. Math. Res. Not.},
    volume = {2015},
    number = {7},
    pages = {1848-1901},
    year = {2014},
    month = {01},
}

@article{lnsss16,
    author = {C.~Lenart and S.~Naito and D.~Sagaki and A.~Schilling and M.~Shimozono},
    title = {A Uniform Model for {K}irillov–{R}eshetikhin Crystals {II}. Alcove Model, Path Model, and {P}$=${X}},
    journal = {Int. Math. Res. Not.},
    volume = {2017},
    number = {14},
    pages = {4259-4319},
    year = {2016},
    month = {07},
    issn = {1073-7928},
    doi = {10.1093/imrn/rnw129},
    url = {https://doi.org/10.1093/imrn/rnw129},
    eprint = {https://academic.oup.com/imrn/article-pdf/2017/14/4259/18767656/rnw129.pdf},
}

@article{lnsss17,
	author = {C.~Lenart and S.~Naito and D.~Sagaki and A.~Schilling and M.~Shimozono},
	da = {2017/12/01},
	doi = {10.1007/s00031-017-9421-1},
	id = {LENART2017},
	isbn = {1531-586X},
	journal = {Transform. Groups},
	number = {4},
	pages = {1041--1079},
	title = {A UNIFORM MODEL FOR {K}IRILLOV--{R}ESHETIKHIN CRYSTALS {III}: NONSYMMETRIC {M}ACDONALD POLYNOMIALS AT $t = 0$ AND {D}EMAZURE CHARACTERS},
	ty = {JOUR},
	url = {https://doi.org/10.1007/s00031-017-9421-1},
	volume = {22},
	year = {2017}}

@article {kraskiewiczweyman01,
    AUTHOR = {W.~Kra\'skiewicz and J.~Weyman},
     TITLE = {Algebra of coinvariants and the action of a {C}oxeter element},
   JOURNAL = {Bayreuth. Math. Schr.},
  FJOURNAL = {Bayreuther Mathematische Schriften},
    VOLUME = {63},
      YEAR = {2001},
     PAGES = {265--284},
      ISSN = {0172-1062},
   MRCLASS = {20C30 (05E10 17B01)},
  MRNUMBER = {1867283},
MRREVIEWER = {Christine\ Bessenrodt},
}

@article {LLT94,
    AUTHOR = {A.~Lascoux and B.~Leclerc and J.Y.~Thibon},
     TITLE = {Green polynomials and {H}all-{L}ittlewood functions at roots
              of unity},
   JOURNAL = {European J. Combin.},
  FJOURNAL = {European Journal of Combinatorics},
    VOLUME = {15},
      YEAR = {1994},
    NUMBER = {2},
     PAGES = {173--180},
      ISSN = {0195-6698,1095-9971},
   MRCLASS = {05E05 (20C30)},
  MRNUMBER = {1261063},
MRREVIEWER = {A.\ O.\ Morris},
       DOI = {10.1006/eujc.1994.1019},
       URL = {https://doi.org/10.1006/eujc.1994.1019},
}

@article{brauneretal25,
      title={Crystal skeletons: Combinatorics and axioms}, 
      author={S.~Brauner and S.~Corteel and Z.~Daugherty and A.~Schilling},
      journal={arXiv:2503.14782},
      year={2025},
      archivePrefix={arXiv},
      primaryClass={math.CO},
      url={https://arxiv.org/abs/2503.14782}, 
}

@article{LS1978,
  author    = {A.~Lascoux and M.~Sch{\"u}tzenberger},
  title     = {Sur une conjecture de {H}. {O}. {F}oulkes},
  journal   = {C. R. Acad. Sci. Paris Sér. A-B},
  volume    = {286},
  year      = {1978},
  pages     = {A323--A324}
}

@incollection{LS1984,
  author    = {A.~Lascoux and M.~Sch{\"u}tzenberger},
  title     = {Tableaux and the {L}ittlewood--{R}ichardson Rule},
  booktitle = {Combinatorics and Algebra},
  series    = {Contemp. Math.},
  volume    = {34},
  pages     = {161--190},
  publisher = {Amer. Math. Soc.},
  year      = {1984}
}

@article{NY1997,
  author    = {A.~Nakayashiki and Y.~Yamada},
  title     = {Kostka polynomials and energy functions in solvable lattice models},
  journal   = {Selecta Math.},
  volume    = {3},
  number    = {4},
  pages     = {547--599},
  year      = {1997},
  doi       = {10.1007/s000290050013}
}

@article{Shimozono1995,
  author    = {M.~Shimozono},
  title     = {A cyclage poset structure for {L}ittlewood--{R}ichardson tableaux},
  journal   = {European J. Combin.},
  volume    = {16},
  number    = {5},
  pages     = {447--462},
  year      = {1995},
  doi       = {10.1016/0195-6698(95)90002-0}
}

@article {Iijima13,
    AUTHOR = {K.~Iijima},
     TITLE = {A {$q$}-multinomial expansion of {LLT} coefficients and
              plethysm multiplicities},
   JOURNAL = {European J. Combin.},
    VOLUME = {34},
      YEAR = {2013},
    NUMBER = {6},
     PAGES = {968--986},
      ISSN = {0195-6698,1095-9971},
       DOI = {10.1016/j.ejc.2013.01.007},
       URL = {https://doi.org/10.1016/j.ejc.2013.01.007},
}

@article{mgt22,
      title={Splitting the square of homogeneous and elementary functions into their symmetric and anti-symmetric parts}, 
      author={F.~Maas-Gariépy and E.~Tétreault},
      year={2022},
      journal={arXiv:2203.08279},
      eprint={2203.08279},
      archivePrefix={arXiv},
      primaryClass={math.CO},
      url={https://arxiv.org/abs/2203.08279}, 
}

@article{carre1995,
  title={Splitting the square of a {S}chur function into its symmetric and antisymmetric parts},
  author={C.~Carr{\'e} and B.~Leclerc},
  journal={J. Algebraic Combin.},
  volume={4},
  number={3},
  pages={201--231},
  year={1995},
  publisher={Springer}
}
\end{document}